\documentclass[11pt]{article}
\usepackage{amsfonts}
\usepackage{amsmath}
\usepackage{amssymb}
\usepackage{amsthm}
\usepackage{enumerate}
\usepackage[pdftex]{graphicx}
\usepackage{subcaption}
\usepackage{url}
\usepackage[ruled,algosection,algo2e]{algorithm2e}

\usepackage{fullpage}

\newtheorem{Theorem}{Theorem}[section]
\newtheorem{Lemma}[Theorem]{Lemma}
\newtheorem{Corollary}[Theorem]{Corollary}
\newtheorem{Observation}[Theorem]{Observation}
\newtheorem{Proposition}[Theorem]{Proposition}
\newtheorem*{introThm}{Theorem}

\theoremstyle{definition}
\newtheorem{Definition}[Theorem]{Definition}
\newtheorem{Example}[Theorem]{Example}
\newtheorem{Remark}[Theorem]{Remark}

\newcommand{\bv}[1]{\mathbf{#1}}
\newcommand{\defbf}[1]{\emph{#1}}

\newcommand{\force}{\rightarrow}
\newcommand{\dcup}{\mathbin{\dot{\cup}}}

\newcommand{\ellstar}{\ell_{\star}}
\newcommand{\astar}{a_{\star}}
\newcommand{\ellstarhat}{\hat{\ell}_{\star}}

\newcommand{\blowup}[2]{#1^{(#2)}}

\newcommand{\zr}[1]{Z_{[#1]}}

\title{Fractional Zero Forcing via Three-color Forcing Games}
\author{
Leslie Hogben\thanks{American Institute of
Mathematics, 600 E. Brokaw Rd., San Jose, CA 95112, USA. \texttt{hogben@aimath.org.}}\phantom{$^*$}\thanks{Department of Mathematics, Iowa State University, Ames, IA 50011, USA. \texttt{\{LHogben, kpalmow,\break myoung\}@iastate.edu.}}
\and Kevin F. Palmowski\footnotemark[2]
\and David E. Roberson\thanks{Division of Mathematical Sciences, Nanyang Technological University, SPMS-MAS-03-01, 21 Nanyang Link, Singapore 637371. \texttt{droberson@ntu.edu.sg.}} 
\and Michael Young\footnotemark[2]
}
\date{\today}

\begin{document}

%%%%%%%%%%%%%%%%%%%%%%%%%%%%%%%%%%%%%%%%%%%%%%%%%%%%%%

\maketitle

\begin{abstract} 
An $r$-fold analogue of the positive semidefinite zero forcing process that is carried out on the $r$-blowup of a graph is introduced and used to define the fractional positive semidefinite forcing number. Properties of the graph blowup when colored with a fractional positive semidefinite forcing set are examined and used to define a three-color forcing game that directly computes the fractional positive semidefinite forcing number of a graph. We develop a fractional parameter based on the standard zero forcing process and it is shown that this parameter is exactly the skew zero forcing number with a three-color approach. This approach and an algorithm are used to characterize graphs whose skew zero forcing number equals zero.
\end{abstract}

{\small


}

%%%%%%%
% !TeX root = frac-zf-master.tex

%%%%%%%%%%%%%%%%%%%%%%%%%%%%%%%%%%%%%%%%%%%%%%%%%%%%%%
\section{Introduction} \label{HPRY-sec-intro-main}

This paper studies fractional versions (in the spirit of \cite{HPRY-fgt}) of the standard and positive semidefinite zero forcing numbers and introduces three-color forcing games to compute these parameters.  The three-color approach allows simpler proofs of some results and yields new results about existing parameters (see, e.g., Section \ref{HPRY-sec-leaf-stripping}).

The zero forcing process was introduced independently in \cite{HPRY-AIM08} as a method of forcing zeros in a null vector of a symmetric matrix described by a graph, which yields an upper bound to the nullity of the matrix, and in \cite{HPRY-BG07} for control of quantum systems. There are potential applications to the spread of rumors or diseases (see, e.g., \cite{HPRY-ZFQC}); one of the original names of zero forcing was ``graph infection." Despite the fact that when studied as a graph parameter there are no zeros involved, the name ``zero forcing number" has become the standard term in the literature. The original zero forcing number has since spawned numerous variants (see, e.g., \cite{HPRY-BBFHHSvdDvdH10, HPRY-param, HPRY-IMAISU10-skew}). The speed with which the zero forcing process colors all vertices has also been studied (see, e.g., \cite{HPRY-proptime, HPRY-nathanThesis}).

%%%%%%%%%%%%%%%%%%%%%%%%%%%%%%
\subsection{Zero forcing games} \label{HPRY-sec-intro-ZF}

In this section, we introduce zero forcing, which can be described as a coloring game \cite{HPRY-param}, and the terminology used. Abstractly, a \defbf{forcing game} is a type of coloring game that is played on a simple graph $G$. First, a ``target color," typically blue or dark blue, is designated. Each vertex of the graph is then colored the target color, white, or possibly some other color (in prior work, only white and the target color have been used). A \defbf{forcing rule} is chosen: this is a rule that describes the conditions under which some vertex can cause another vertex to change to the target color. If vertex $u$ causes a neighboring vertex $w$ to change color, we say that $u$ \defbf{forces} $w$ and write $u \force w$. The forcing rule is repeatedly applied until no more forces can be performed, at which point the game ends; the coloring at the end is called the \defbf{final coloring}. An ordered list of the forces performed is referred to as a \defbf{chronological list of forces}. Note that there is usually some choice as to which forces are performed, as well as the order in which these forces occur. As such, a single forcing set may generate many different chronological lists of forces; however, the final coloring is unique for all of the games discussed herein. If the graph is totally colored with the target color at the end of the game, then we say that $G$ has been \defbf{forced}. The goal of the game is to force $G$. If this is possible, then the initial set of non-white vertices is called a \defbf{forcing set}.

The \defbf{(standard) zero forcing game} uses only the colors blue (the target color) and white. The \defbf{(standard) zero forcing rule} is as follows:
\begin{quote}
\begin{center}
If $w$ is the only white neighbor of a blue vertex $u$, then $u$ can force $w$.
\end{center}
\end{quote}
A \defbf{(standard) zero forcing set} is an initial set of blue vertices that can force $G$ using this rule. The \defbf{(standard) zero forcing number} of $G$, denoted $Z(G)$, is the minimum cardinality of a zero forcing set for $G$. We present an illustrative example in Figure \ref{HPRY-fig-intro-ex-zf}.

\begin{figure}[h]
\begin{center}
	\begin{subfigure}[b]{0.22\textwidth}
	\begin{center}
		\includegraphics[scale=0.8]{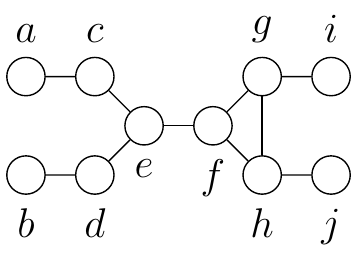} \qquad 
		\caption{Graph $G$}
		\label{HPRY-fig-intro-ex-G}
	\end{center}
	\end{subfigure}
	\quad
	\begin{subfigure}[b]{0.22\textwidth}
	\begin{center}
		\includegraphics[scale=0.8]{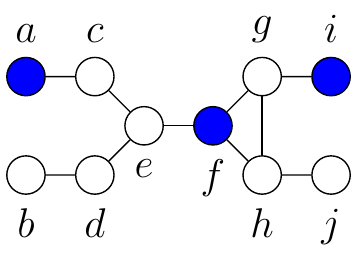} \qquad 
		\caption{Initial forcing set}
		\label{HPRY-fig-intro-ex-zf-a}
	\end{center}
	\end{subfigure}
	\quad
	\begin{subfigure}[b]{0.22\textwidth}
	\begin{center}
		\includegraphics[scale=0.8]{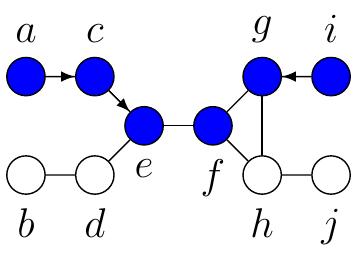} \qquad 
		\caption{First three forces}
		\label{HPRY-fig-intro-ex-zf-b}
	\end{center}
	\end{subfigure}
	\quad
	\begin{subfigure}[b]{0.22\textwidth}
	\begin{center}
		\includegraphics[scale=0.8]{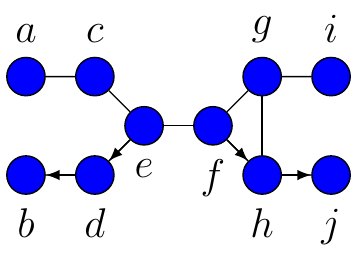}
		\caption{Final forces}
		\label{HPRY-fig-intro-ex-zf-c}
	\end{center}
	\end{subfigure}
\caption{Standard zero forcing game example}
\label{HPRY-fig-intro-ex-zf}
\end{center}
\end{figure}

From this point forward, we will omit the word ``standard" when referring to the standard zero forcing game, its forcing rule, or zero forcing sets whenever there is no risk of ambiguity.

The \defbf{positive semidefinite zero forcing game} is a modification of the zero forcing game used to force zeros in a null vector of a positive semidefinite matrix described by a graph \cite{HPRY-BBFHHSvdDvdH10}. Like the zero forcing game, positive semidefinite zero forcing uses only the colors blue (target) and white. The \defbf{positive semidefinite zero forcing rule} is the same as the standard zero forcing rule, except that this rule also features a \defbf{disconnect rule}:
\begin{quote}
\begin{center}
Remove all blue vertices from the graph, leaving a set of connected components. To each connected component (of white vertices) in turn, add the blue vertices, the edges among the blue vertices, and any edges between the blue vertices and that component, and perform forces via the standard rule: If $w$ is the only white neighbor of a blue vertex $u$ in this induced subgraph, then $u$ can force $w$.
\end{center}
\end{quote}
It is not assumed that disconnection occurs; if there is only one component, then we simply force via the standard forcing rule. If disconnection does occur, then after the force the graph is ``reassembled" prior to applying the rule again. As one would expect, a \defbf{positive semidefinite zero forcing set} is an initial set of blue vertices that can force $G$ using this rule, and the \defbf{positive semidefinite zero forcing number} of $G$, denoted $Z^+(G)$, is the minimum cardinality of a positive semidefinite zero forcing set for $G$. In Figure \ref{HPRY-fig-intro-ex-zfp} we illustrate the positive semidefinite zero forcing process on the graph from Figure \ref{HPRY-fig-intro-ex-G}.

\begin{figure}[h]
\begin{center}
	\begin{subfigure}[b]{0.3\textwidth}
	\begin{center}
		\includegraphics[scale=0.8]{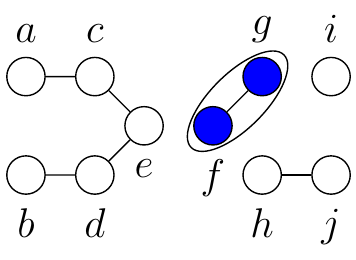} \qquad 
		\caption{Connected components}
		\label{HPRY-fig-intro-ex-zfp-b}
	\end{center}
	\end{subfigure}
	\quad
	\begin{subfigure}[b]{.3\textwidth}
	\begin{center}
		\includegraphics[scale=0.75]{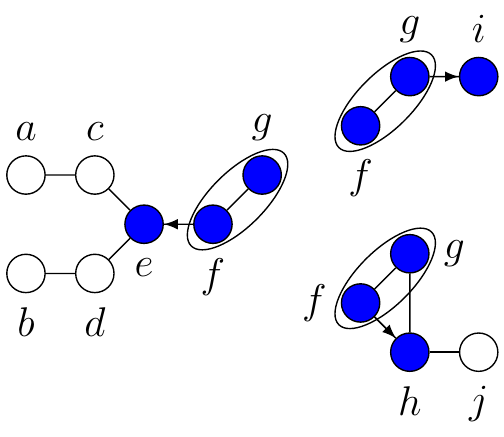} \qquad 
		\caption{Forcing in each component}
		\label{HPRY-fig-intro-ex-zfp-c}
	\end{center}
	\end{subfigure}
	\quad
	\begin{subfigure}[b]{0.3\textwidth}
	\begin{center}
		\includegraphics[scale=0.8]{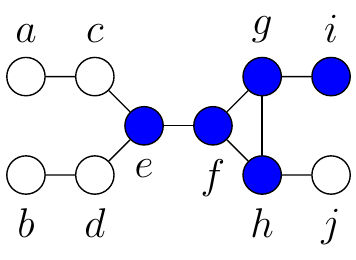} \qquad 
		\caption{Reassembled graph}
		\label{HPRY-fig-intro-ex-zfp-d}
	\end{center}
	\end{subfigure}
\caption{Positive semidefinite zero forcing game example (first steps)}
\label{HPRY-fig-intro-ex-zfp}
\end{center}
\end{figure}

The \defbf{skew zero forcing game}, another variant on zero forcing that uses the colors white and blue (target), was first considered in \cite{HPRY-IMAISU10-skew} to force zeros in a null vector of a skew symmetric matrix described by a graph. The \defbf{skew zero forcing rule} is as follows: 
\begin{quote}
\begin{center}
If $w$ is the only white neighbor of any vertex $u$, then $u$ can force $w$.
\end{center}
\end{quote}
Skew zero forcing removes the standard requirement that the forcing vertex $u$ be blue; as a result, skew zero forcing allows \defbf{white vertex forcing}, i.e., a white vertex is allowed to force its only white neighbor. A \defbf{skew zero forcing set} is an initial set of blue vertices that can force $G$ using this rule, and the \defbf{skew zero forcing number} of $G$, denoted $Z^-(G)$, is the minimum cardinality of a skew zero forcing set for $G$. Figure \ref{HPRY-fig-intro-ex-zfm} demonstrates skew zero forcing; notice that the initial forcing set contains no blue vertices. 

\begin{figure}[h]
\begin{center}
	\begin{subfigure}[b]{0.3\textwidth}
	\begin{center}
		\includegraphics[scale=0.8]{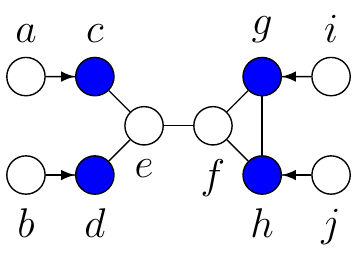} \qquad 
		\caption{First four forces}
		\label{HPRY-fig-intro-ex-zfm-a}
	\end{center}
	\end{subfigure}
	\quad
	\begin{subfigure}[b]{0.3\textwidth}
	\begin{center}
		\includegraphics[scale=0.8]{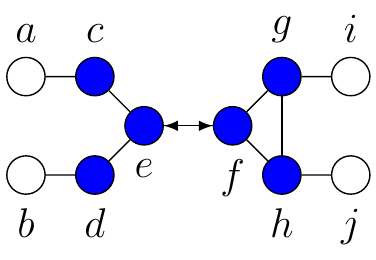} \qquad 
		\caption{Fifth and sixth forces}
		\label{HPRY-fig-intro-ex-zfm-b}
	\end{center}
	\end{subfigure}
	\quad
	\begin{subfigure}[b]{0.3\textwidth}
	\begin{center}
		\includegraphics[scale=0.8]{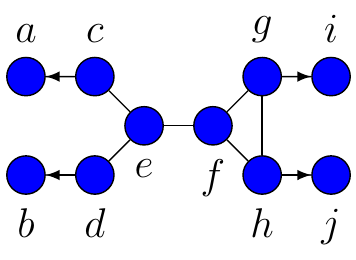}
		\caption{Final forces}
		\label{HPRY-fig-intro-ex-zfm-c}
	\end{center}
	\end{subfigure}
\caption{Skew zero forcing game example}
\label{HPRY-fig-intro-ex-zfm}
\end{center}
\end{figure}

%%%%%%%%%%%%%%%%%%%%%%%%%%%%%%%
\subsection{Motivation and method}

This paper focuses on fractional versions of the standard and positive semidefinite zero forcing numbers. We first present the construction of fractional chromatic number found in \cite{HPRY-fgt} as an example of the method used to define a fractional graph parameter. A \defbf{proper coloring} of a graph $G$ is an assignment of colors to the vertices of $G$ such that adjacent vertices receive different colors. The \defbf{chromatic number} of $G$, denoted $\chi(G)$, is the least number of colors required to properly color $G$. We can generalize a proper coloring of $G$ using $c$ colors to a \defbf{proper $r$-fold coloring with $c$ colors}, or a $c$:$r$-coloring: from a total of $c$ colors, we assign $r$ colors to each vertex of $G$ such that adjacent vertices receive disjoint sets of colors. The \defbf{$r$-fold chromatic number} of $G$, denoted $\chi_r(G)$, is the smallest value of $c$ such that G has a $c$:$r$-coloring; we emphasize that to compute $\chi_r(G)$ we fix $r$ and minimize the value of $c$. The \defbf{fractional chromatic number} of $G$ is then defined as
\[ \chi_f(G) = \inf_{r \in \mathbb{N}} \left\{ \frac{\chi_r(G)}{r} \right\} . \]
The interested reader is referred to \cite{HPRY-fgt} for an in-depth treatment of fractional chromatic number, as well as other fractional graph parameters. For this paper, defining an $r$-fold version of a graph parameter and then defining the fractional parameter as the infimum of the ratios of the $r$-fold parameter to $r$ are key ideas.

Suppose that $G$ is a simple graph on $n$ vertices with $V(G) = [n]$. We say that a symmetric matrix $A \in \mathbb{C}^{nr \times nr}$ \defbf{$r$-fits} $G$ if, after partitioning $A$ as a block $n \times n$ matrix, block $A_{ii} = I_r$ for each $i$ and for all $i, j$ with $i \neq j$, block $A_{ij} = 0_{r \times r}$ if and only if $ij \notin E(G)$ \cite{HPRY-HPRS15}. While there may be many such matrices for a given graph, the following result shows that certain structure can be chosen.

%%%%%%%%%%%%%%%
\begin{Proposition} \label{HPRY-prop-rfits-structure}
Suppose that $A \in \mathbb{C}^{nr \times nr}$ $r$-fits a graph $G$ on $n$ vertices. We can construct a unitary matrix $U$ such that $U^* A U$ $r$-fits $G$ and if $ij \in E(G)$, then every entry of block $\left(U^* A U \right)_{ij}$ is nonzero.
\end{Proposition}
\begin{proof}
Assume that $V(G) = [n]$ and partition $A = [ A_{ij} ]$ as an $n \times n$ block matrix with $A_{ij} \in \mathbb{C}^{r \times r}$. By definition, we have $A_{ii} = I_r$ for each $i \in [n]$, and for $i,j \in [n]$ with $i \neq j$ we have $A_{ij} = 0_{r \times r}$ if and only if $ij \notin E(G)$.

For each $i \in [n]$, let $U_i \in \mathbb{C}^{r \times r}$ be a random unitary matrix with $U_i$ and $U_j$ chosen independently if $i \neq j$. Define $U = \operatorname{blockdiag}(U_1, \ldots, U_n)$ and let $C = U^* A U$. Partitioning $C$ conformally with $A$, we have $C_{ij} = U_i^* A_{ij} U_j$. Notice that $C_{ii} = U_i^* I_r U_i = I_r$ and if $ij \notin E(G)$ (for $i \neq j$), then $C_{ij} = U_i^* 0_{r \times r} U_j = 0_{r \times r}$.

Suppose $ij \in E(G)$ and consider the product $A_{ij} U_j$. Since $U_j$ is random, with high probability no column of $U_j$ lies in $\ker A_{ij}$, so no column of $A_{ij} U_j$ is a zero vector. Let $\bv{z}$ be any column of $A_{ij} U_j$ (so with high probability $\bv{z} \neq \bv{0}$) and consider $(U_i^* \bv{z})_k$. If $(U_i^* \bv{z})_k = 0$, then $\bv{z}$ is orthogonal to the $k^{th}$ column of $U_i$. Since $U_i$ is a random unitary matrix, with high probability this does not happen. We conclude that if $ij \in E(G)$, then with high probability no entry of $C_{ij}$ is zero. Thus there exists a matrix that $r$-fits $G$ and has the desired structure.
\end{proof}
%%%%%%%%%%%%%%%

Let $G$ be a graph and choose $r \in \mathbb{N}$. The \defbf{$r$-blowup} of $G$ is the graph $\blowup{G}{r}$ constructed by replacing each vertex of $u \in V(G)$ with an independent set of $r$ vertices, denoted $R_u$, and replacing each edge $uw \in E(G)$ by the edges of a complete bipartite graph on partite sets $R_u$ and $R_w$.%
\footnote{Given graphs $G$ and $H$, the \defbf{lexicographic product} of $G$ with $H$, denoted $G \times_L H$, is the graph with $V(G \times_L H) = V(G) \times V(H)$ and $(g,h)(i,j) \in E(G \times_L H)$ if $gi \in E(G)$ or if $g = i$ and $hj \in E(H)$. We can also define the $r$-blowup of $G$ as $\blowup{G}{r} = G \times_L \overline{K_r}$, where $\overline{K_r}$ denotes the empty graph on $r$ vertices.}
We call the set $R_u$ a \defbf{cluster}. Note that $V(\blowup{G}{r}) = \bigcup_{u \in V(G)} R_u$ and if $uw \in E(G)$ then every vertex of $R_u$ is adjacent to every vertex of $R_w$ in $\blowup{G}{r}$.

Suppose that $A \in \mathbb{C}^{nr \times nr}$ is positive semidefinite and $r$-fits a graph $G$ on $n$ vertices with $V(G) = [n]$. As a result of Proposition \ref{HPRY-prop-rfits-structure} (by replacing $A$ with $U^*AU$), we can assume that if $ij \in E(G)$, then block $A_{ij}$ has no zero entries. Consider the graph of such a matrix $A$, namely, the simple graph with vertex set $[nr]$ and with an edge between vertices $k$ and $\ell$ if $k \neq \ell$ and the entry in row $k$ and column $\ell$ of $A$ is nonzero. Since $A_{ii} = I_r$, the vertices of $G$ will map to independent sets (clusters) of size $r$; let $R_i$ denote the cluster associated with vertex $i \in V(G)$. Since each entry of $A_{ij}$ is nonzero, every vertex in $R_i$ will be adjacent to every vertex in $R_j$, and vice versa. Hence the graph of $A$ is exactly $\blowup{G}{r}$, the $r$-blowup of $G$.

The positive semidefinite zero forcing number of a graph is an upper bound on the maximum positive semidefinite nullity of the graph, which equals the order of the graph minus its minimum positive semidefinite rank \cite{HPRY-BBFHHSvdDvdH10, HPRY-HLA2ch46}. The authors of \cite{HPRY-HPRS15} define an $r$-fold analogue of minimum positive semidefinite rank and use this new parameter to define fractional minimum positive semidefinite rank. A key element of this treatment is that the $r$-fold minimum positive semidefinite rank of a graph can be expressed as the rank of a positive semidefinite matrix that $r$-fits the graph \cite[Theorem 3.9]{HPRY-HPRS15}. Our previous discussion allows us to assume that the graph of such a matrix is $\blowup{G}{r}$.

As mentioned in Section \ref{HPRY-sec-intro-ZF}, playing the positive semidefinite zero forcing game can be interpreted as forcing zeros in a null vector of a positive semidefinite matrix whose graph is $G$, hence the connection to maximum positive semidefinite nullity and minimum positive semidefinite rank. Since the $r$-fold minimum positive semidefinite rank is defined in terms of matrices that $r$-fit the original graph, an $r$-fold analogue of positive semidefinite zero forcing number would naturally be associated with a game played on the graph of a positive semidefinite matrix that $r$-fits $G$. To this end, our $r$-fold forcing parameters will be defined in terms of forcing games played on $\blowup{G}{r}$.

%%%%%%%%%%%%%%%%%%%%%%%%%%%%%%
\subsection{Definitions and notation} \label{HPRY-sec-intro-defs}

Throughout this paper, all graphs are simple. We use $|G|$ to denote the order of a graph $G$, i.e., $|G| = |V(G)|$. If $G$ is a graph and $S \subseteq V(G)$, then $G[S]$ denotes the subgraph of $G$ \defbf{induced} by $S$, namely, the graph with $V(G[S]) = S$ and $E(G[S]) = \{ uv \in E(G) : u, v \in S\}$. We use $G - S$ as shorthand for the induced subgraph $G[V(G) \setminus S]$. The \defbf{neighborhood} of a vertex $u \in V(G)$, denoted $N(u)$, is the set of vertices adjacent to $u$. The \defbf{degree} of a vertex $u$, $\operatorname{deg}(u)$, is the number of neighbors of $u$, i.e., $|N(u)|$. A \defbf{leaf} is a vertex of degree one. We use $\delta(G)$ to denote the minimum of the degrees of the vertices of $G$. We write $S \dcup T$ to denote the union of disjoint sets $S$ and $T$.

Throughout, $B$ will be used to denote a set of blue vertices associated with a two-color forcing game. We emphasize that in a two-color forcing game the target color is blue. For three-color forcing games, we use two non-white colors: dark blue, which is our target color, and light blue. A set of colored vertices associated with a three-color forcing game with be denoted by $\mathcal{B}$. Given such a set $\mathcal{B}$, we let $\mathcal{D}$ be the set of dark blue vertices and $\mathcal{L}$ be the set of light blue vertices. Since $\mathcal{D} \cap \mathcal{L} = \emptyset$, we have $\mathcal{B} = \mathcal{D} \dcup \mathcal{L}$. While $\mathcal{B}$ is a set, we will abuse notation and write $\mathcal{B} = (\mathcal{D}, \mathcal{L})$ to emphasize the decomposition of $\mathcal{B}$ into its component sets.

%%%%%%%%%%%%%%%%%%%%%%%%%%%%%%
\subsection{Contribution and organization of the paper}

In Section \ref{HPRY-sec-psd-zf} we introduce and examine the fractional positive semidefinite forcing number of a graph. An $r$-fold extension of the positive semidefinite zero forcing number, based on graph blowups, is introduced and used to define the fractional positive semidefinite forcing number of a graph $G$, denoted $Z_f^+(G)$. We also introduce a three-color forcing game played on $G$ called the fractional positive semidefinite forcing game and prove a main result of that section:
\begin{introThm}[Theorem \ref{HPRY-thm-zfp-equals-zhat}]
For any graph $G$, $Z_f^+(G)$ is the minimum number of dark blue vertices in a (three-color) fractional positive semidefinite forcing set for $G$.
\end{introThm}
This result allows us to determine the fractional positive semidefinite forcing number of a graph by playing the fractional positive semidefinite forcing game, as opposed to computation via the $r$-fold approach. We prove numerous results pertaining to fractional positive semidefinite forcing number and the structure of optimal fractional positive semidefinite forcing sets and apply these results to compute the fractional positive semidefinite forcing number for some common graph families. We also prove that any graph has an ordinary (two-color) minimum positive semidefinite zero forcing set such that the first force in the forcing process can be done without using the disconnect rule.

In Section \ref{HPRY-sec-skew} we introduce a three-color forcing game that is equivalent to the skew zero forcing game. The three-color approach is used to prove numerous results pertaining to skew zero forcing. We define an $r$-fold analogue of the (standard) zero forcing game and using this to define the fractional forcing number of a graph, denoted $Z_f(G)$. A main result of that section shows that skew zero forcing number and fractional zero forcing number of a graph are the same:
\begin{introThm}[Theorem \ref{HPRY-thm-zf-equals-skew}]
For any graph $G$, $Z_f(G) = Z^-(G)$.
\end{introThm}
We conclude the section by introducing an algorithm that is used to characterize graphs that satisfy $Z^-(G) = 0$.

%%%%%%%
% !TeX root = ./frac-zf-master.tex

%%%%%%%%%%%%%%%%%%%%%%%%%%%%%%%%%%%%%%%%%%%%%%%%%%%%%%
\section{Fractional positive semidefinite forcing} \label{HPRY-sec-psd-zf}

In this section, we introduce the $r$-fold and fractional positive semidefinite forcing numbers of a graph, as well as a three-color forcing game that relates to the fractional parameter.

%%%%%%%%%%%%%%%%%%%%
\subsection{The $r$-fold positive semidefinite forcing game and fractional positive\break semidefinite forcing number} \label{HPRY-sec-rfold-PSD-intro}

Let $G$ be a graph and for $r \in \mathbb{N}$ consider the following \defbf{$r$-fold positive semidefinite forcing game}, which is a two-color forcing game played on $\blowup{G}{r}$. As in any forcing game, we initially color some set $B \subseteq V(\blowup{G}{r})$ blue and then try to force $\blowup{G}{r}$ through repeated application of the following $r$-fold positive semidefinite forcing rule:

%%%%%%%%%%%%%
\begin{Definition}[$r$-fold positive semidefinite forcing rule]
Let $B_t$ denote the set of blue vertices of $\blowup{G}{r}$ at some step $t$ of the $r$-fold positive semidefinite forcing process%
\footnote{We caution the reader that a chronological list of forces is not a propagating process and $B_t$ here has different meaning than that used in the study of propagation.} 
and let $W_1, \ldots, W_h$ denote the sets of vertices of the connected components of $\blowup{G}{r} - B_t$. If $u \in B_t$ and $|N(u) \cap W_i| \leq r$, then $u$ can force $N(u) \cap W_i$, i.e., all white neighbors of $u$ in $\blowup{G}{r}[B_t \cup W_i]$ can be simultaneously colored blue.
\end{Definition}
%%%%%%%%%%%%%

The $r$-fold positive semidefinite forcing game can be thought of as a generalization of the positive semidefinite zero forcing game: instead of forcing one white neighbor in a component after applying the disconnect rule, a vertex forces up to $r$ white neighbors in a component. This is a positive semidefinite analogue of the $r$-forcing process described in \cite{HPRY-davila-kforcing}, but we apply this process only to the blowup of the graph.

If $\blowup{G}{r}$ can be forced, then the initial set of blue vertices is called an \defbf{$r$-fold positive semidefinite (PSD) forcing set} for $G$. An $r$-fold PSD forcing set $B$ is \defbf{minimum} if there is no $r$-fold PSD forcing set of smaller cardinality than $B$. The cardinality of a minimum $r$-fold PSD forcing set is called the \defbf{$r$-fold positive semidefinite forcing number} of $G$ and is denoted $\zr{r}^+(G)$. We define the \defbf{fractional positive semidefinite forcing number} of a graph $G$ as
\[ Z_f^+(G) = \inf_{r \in \mathbb{N}} \left\{ \frac{\zr{r}^+(G)}{r} \right\}. \]

Note that $\blowup{G}{1} = G$ and a $1$-fold PSD forcing set is exactly a positive semidefinite zero forcing set, so $\zr{1}^+(G) = Z^+(G)$. Any positive semidefinite zero forcing set $B$ can be converted into an $r$-fold PSD forcing set (for $r \geq 2$) by the following rule: If $u \in B$, then color every vertex in $R_u \in V(\blowup{G}{r})$ blue. This creates an $r$-fold PSD forcing set that contains $r \cdot Z^+(G)$ blue vertices, so $\zr{r}^+(G) \leq r \cdot Z^+(G) = r \cdot \zr{1}^+(G)$. We conclude that $Z_f^+(G) \leq Z^+(G)$ and that
\[ Z_f^+(G) = \inf_{r \geq 2} \left\{ \frac{\zr{r}^+(G)}{r} \right\}. \]

%%%%%%%%%%%%%%%%%%%%%%%%%%%%%%%%%%%%%%%%
\subsection{Global interpretation of $r$-fold positive semidefinite forcing} \label{HPRY-sec-global-PSD}

In this section, we assume that $r \geq 2$ and utilize the global structure of a graph $r$-blowup, namely, clusters joined by edges. Three specific types of cluster are of particular interest. An \defbf{All cluster} is a cluster in which all vertices are colored blue. A \defbf{One cluster} is a cluster in which exactly one vertex is colored blue and the rest are colored white. A \defbf{None cluster} is a cluster in which all vertices are colored white. We define a \defbf{All-One-None (minimum) $r$-fold positive semidefinite forcing set} $B$ for a graph $G$ to be a (minimum) $r$-fold PSD forcing set in which each cluster of $\blowup{G}{r}$ is either an All, One, or None cluster when $\blowup{G}{r}$ is colored with $B$. For the sake of brevity, we will hereafter shorten All-One-None to AON.

We say that a cluster $R_u$ is \defbf{forced into} when any vertex in $R_u$ is forced. Once a cluster changes from a non-All to an All cluster, we say that the cluster has been \defbf{forced}. Any cluster that is forced into becomes an All cluster after the forcing operation, so forcing into a cluster and forcing the cluster are equivalent.

%%%%%%%%%%%%%%%
\begin{Remark} \label{HPRY-rmk-backforce-ok}
At some stage of the $r$-fold positive semidefinite forcing process using a particular chronological list of forces, let $B_t$ denote the set of blue vertices in $\blowup{G}{r}$. Assume that $R_u \not\subseteq B_t$ for some $u \in V(G)$. Suppose that the next force in the process is done by $x \in R_u$, so $x$ has at most $r$ white neighbors. Since $R_u \not\subseteq B_t$, there exists at least one white vertex $w \in R_u$. Because $x$ and $w$ have the same neighbors and $w$ is white, all white neighbors of $x$ are connected through $w$ and lie in the same connected component. Hence, after $x$ forces, all neighbors of every vertex in $R_u$ must be blue, so without loss of generality $R_u$ can be forced in the next step of the forcing process.
\end{Remark}
%%%%%%%%%%%%%%%

%%%%%%%%%%%%%%%
\begin{Definition}
If at any stage of the $r$-fold positive semidefinite forcing process a vertex in any partially-filled cluster performs a force, then that cluster can itself be forced at the next forcing step. We refer to this process as \defbf{backforcing}.
\end{Definition}
%%%%%%%%%%%%%%%'

Remark \ref{HPRY-rmk-backforce-ok} asserts that requiring backforcing does not affect whether or not a set is an $r$-fold PSD forcing set, so we will always assume that backforcing is used when performing the $r$-fold positive semidefinite forcing process.

%%%%%%%%%%%%%%%
\begin{Definition}
Let $R_{u_1}, R_{u_2}, \ldots, R_{u_m}$ be ``partially-filled" clusters (i.e., no cluster is an All or a None) in $\blowup{G}{r}$ that together contain $pr+q$ blue vertices for some $0 \leq p < m$ and $0 \leq q <r$. We define the process of \defbf{consolidation} as follows: Use $pr$ of the blue vertices to convert $R_{u_1}, \ldots, R_{u_p}$ into All clusters and move the remaining $q$ blue vertices into $R_{u_{p+1}}$.
\end{Definition}
%%%%%%%%%%%%%%%

Our goal for the remainder of this section is to use these tools and definitions to develop an equivalent characterization of the $r$-fold positive semidefinite forcing game that relies only upon a particular type of AON $r$-fold PSD forcing set.

%%%%%%%%%%%%%%%
\begin{Remark} \label{HPRY-rmk-AON-only-one-force}
Suppose that $r \geq 3$. If $B$ is an AON $r$-fold PSD forcing set, then from a global perspective exactly one cluster in $\blowup{G}{r}$ is forced at each step of the forcing process. This is because the vertex that performs the force can only force into One or None clusters, and if this vertex were adjacent to more than one of these (in any combination), then it would have more than $r$ white neighbors and could not actually perform a force.
\end{Remark}
%%%%%%%%%%%%%%%

The case when $r = 2$ is slightly different. In this case, it is possible for a vertex to force two One clusters at the same forcing step (see Example \ref{HPRY-ex-2fold-bad-replication} below). Every $2$-fold PSD forcing set is automatically an AON set, so we cannot claim that if $\blowup{G}{r}$ has a global AON structure, then exactly one cluster will be forced at each forcing step. However, Theorem \ref{HPRY-thm-AON-P} uses consolidation to show that even though every AON PSD forcing set need not have this property, there always exist an AON minimum PSD forcing set and forcing process that do.

%%%%%%%%%%%%%%%
\begin{Theorem} \label{HPRY-thm-AON-P}
Let $G$ be a graph and suppose $r \geq 2$. Then there exists an AON minimum $r$-fold PSD forcing set for $G$. For all $r \geq 3$, exactly one cluster of $\blowup{G}{r}$ will be forced at each step of any forcing process that begins with any such set. For $r = 2$, there exists a forcing process for the set constructed such that exactly one cluster of $\blowup{G}{r}$ is forced at each step.
\end{Theorem}
\begin{proof}
We first consider the case where $r \geq 3$. Let $B$ be a minimum $r$-fold PSD forcing set for $G$ and assume that $B$ is not AON. Write a chronological list of the forces performed using the forcing set $B$, assuming the use of backforcing, and let $B_t$, $t \geq 0$, denote the set of blue vertices after step $t$ of this forcing process, where $B_0 = B$.

Suppose that a vertex $x  \in R_u$ performs a force at step $\ell \geq 1$ of the forcing process and $R_u \not \subseteq B_{\ell-1}$, implying that $R_u$ was not forced into at any step prior to step $\ell$. Since we assume backforcing and $R_u$ contains at least one white vertex, $R_u$ was not used to force any other cluster prior to step $\ell$, and $R_u$ will be forced in step $\ell+1$. Thus if $R_u$ is not a One cluster, we can uncolor every blue vertex in $R_u$ except for $x$ without changing the ability of $x$ to force or the ability of $R_u$ to be backforced at step $\ell+1$; since $R_u$ is not involved in any forces prior to step $\ell$, we can make this change in the original set $B$ and obtain a forcing set with fewer blue vertices, contradicting the assumption that $B$ was a minimum forcing set. We conclude that every cluster in a minimum $r$-fold PSD forcing set that is not an All cluster and contains a vertex that performs a force must be a One cluster.

Now, suppose that at step $\ell \geq 1$ we have $x \force W \subseteq \left( R_{u_1} \cup R_{u_2} \cup \cdots \cup R_{u_m} \right)$ for some $m \geq 2$, where each $R_{u_j}$ contains at least one white vertex. Since $x$ is performing a force, it has at most $r$ white neighbors in the component containing $\bigcup_{j=1}^m R_{u_j}$, so there are at least $r(m-1)$ blue vertices in $\bigcup_{j=1}^m R_{u_j}$. Each cluster $R_{u_j}$ is an All cluster after step $\ell$, and no $R_{u_j}$ was forced into prior to step $\ell$. Since we assume backforcing and each of the $R_{u_j}$ clusters contains at least one white vertex, none of the $R_{u_j}$ clusters contains a vertex that was used to force at a step prior to step $\ell$. Analogous to Remark \ref{HPRY-rmk-backforce-ok}, removing blue vertices from any of the $R_{u_j}$ will not affect the application of the disconnect property, as each $R_{u_j}$ contains at least one white vertex. Similarly, adding blue vertices to convert an $R_{u_j}$ into an All cluster may make available additional disconnects (which we do not use), but these would not affect any previous forces. Therefore, we can consolidate the (at least $r(m-1)$) blue vertices in $\bigcup_{j=1}^m R_{u_j}$ without affecting the ability to perform any previous force.

Without loss of generality, suppose that $R_{u_1}, \ldots, R_{u_{m-1}}$ become All clusters after the consolidation and any remaining blue vertices are left in $R_{u_m}$. After consolidation, the new force at step $\ell$ will be $x \force R_{u_m}$; after this point, the state of the system is the same as it would have been had we not consolidated (i.e., every $R_{u_j}$ is an All cluster), so future forces are unaffected by consolidation. Furthermore, after consolidation, exactly one cluster ($R_{u_m}$) is forced at step $\ell$. Since the consolidation process does not affect any of the forces before or after the force at step $\ell$, we are free to perform the consolidation on the original set $B$ to obtain a new minimum $r$-fold PSD forcing set $\widetilde{B}$ and the sequence of vertices that perform forces remains unchanged. Note that since $\widetilde{B}$ is minimum, $R_{u_m}$ must necessarily be a None cluster.

By repeated application of the consolidation process, we are able to convert every non-One cluster into an All cluster or a None cluster. By Remark \ref{HPRY-rmk-AON-only-one-force}, any AON forcing process for $r \geq 3$ must necessarily consist of forcing only one cluster at each step, which proves the claim for $r \geq 3$.

Now, suppose that $r = 2$. Every minimum $2$-fold PSD forcing set for $G$ is automatically an AON set. Suppose that, at step $\ell \geq 1$ of the forcing process, more than one cluster must be forced. Since any vertex can force at most 2 of its neighbors, it must be the case that two One clusters are forced at this step. For the reasons described in the $r \geq 3$ case, we can consolidate these two One clusters into one All cluster and one None cluster without affecting any previous or future forces; after this consolidation, only one cluster is forced at step $\ell$. As before, we can modify our original minimum forcing set and the result follows for the $r = 2$ case (using the forcing process to which consolidation was applied).
\end{proof}
%%%%%%%%%%%%%%%

We call the type of AON minimum $r$-fold PSD forcing set guaranteed to exist by Theorem \ref{HPRY-thm-AON-P} an \defbf{optimal AON $r$-fold PSD forcing set}. We emphasize that an optimal AON $r$-fold PSD forcing set is minimum by definition, and given an optimal AON $r$-fold PSD forcing set there is a corresponding forcing process in which exactly one cluster is forced at each step. Further, the set of blue vertices at each step of the forcing process associated with an optimal AON $r$-fold PSD forcing set will always create a global AON structure in $\blowup{G}{r}$.

Suppose that $B$ is an AON $r$-fold PSD forcing set for a graph $G$ and color $\blowup{G}{r}$ with $B$. We use $a(B)$ to denote the number of All clusters in $\blowup{G}{r}$ and $\ell(B)$ to denote the number of One clusters in $\blowup{G}{r}$, so $|B| = r \cdot a(B) + \ell(B)$. This new terminology yields a corollary to Theorem \ref{HPRY-thm-AON-P}.

%%%%%%%%%%%%%%%%
\begin{Corollary} \label{HPRY-cor-zr-alt-def}
For every graph $G$ and $r \geq 2$, there exists an optimal AON $r$-fold PSD forcing set for $G$, and for any such set $B$, we have $\zr{r}^+(G) = |B| = r \cdot a(B) + \ell(B)$.
\end{Corollary}
%%%%%%%%%%%%%%%%

%%%%%%%%%%%%%%%%
\begin{Definition}
Let $r, s \geq 2$ with $s \neq r$ and suppose that $B$ is an AON $r$-fold PSD forcing set for $G$. Copy the AON structure of $\blowup{G}{r}$ when colored with $B$ onto $\blowup{G}{s}$ to create a new AON set of blue vertices of cardinality $s \cdot a(B) + \ell(B)$. This process is called \defbf{replication}.
\end{Definition}
%%%%%%%%%%%%%%%%

%%%%%%%%%%%%%%%%
\begin{Remark} \label{HPRY-rmk-replication-P}
Let $B$ be a $2$-fold PSD forcing set for $G$ and suppose that two One clusters are forced simultaneously at some step of the forcing process on $\blowup{G}{2}$. In this case, replicating $B$ onto $\blowup{G}{s}$ for $s > 2$ will not yield a valid forcing set (see Example \ref{HPRY-ex-2fold-bad-replication}, next). However, if $B$ is an optimal AON $2$-fold PSD forcing set, then Theorem \ref{HPRY-thm-AON-P} guarantees that there is a forcing process in which exactly one force occurs at each step, so replication will yield a valid forcing set. As we see in Example \ref{HPRY-ex-not-AON-after-replication}, however, the replicated set may not be minimum and hence not optimal.
\end{Remark}
%%%%%%%%%%%%%%%%

\begin{figure}[h]
\begin{center}
	\begin{subfigure}[b]{0.45\textwidth}
	\begin{center}
		\includegraphics[scale=1]{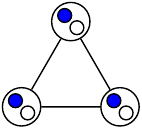} \qquad 
		\caption{(Minimum) AON $2$-fold PSD forcing set}
		\label{HPRY-fig-psd-replication-K3-a}
	\end{center}
	\end{subfigure}
	\quad
	\begin{subfigure}[b]{0.45\textwidth}
	\begin{center}
		\includegraphics[scale=1]{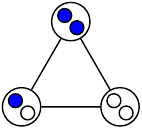} \qquad 
		\caption{Optimal AON $2$-fold PSD forcing set}
		\label{HPRY-fig-psd-replication-K3-b}
	\end{center}
	\end{subfigure}
\caption{AON $2$-fold PSD forcing sets for $K_3$}
\label{HPRY-fig-psd-replication-K3}
\end{center}
\end{figure}

%%%%%%%%%%%%%%%%%%
\begin{Example} \label{HPRY-ex-2fold-bad-replication}
Consider the (minimum) $2$-fold PSD forcing sets for $K_3$ shown in Figure \ref{HPRY-fig-psd-replication-K3}. For simplicity, the edges in the figure represent the complete bipartite graphs between the clusters at their endpoints. The first forcing step in Figure \ref{HPRY-fig-psd-replication-K3-a} would consist of forcing two of the One clusters simultaneously. This set is no longer a forcing set when replicated onto $\blowup{K_3}{s}$ for $s \geq 3$, as each of the blue vertices will have too many white neighbors to perform a force. The optimal AON PSD forcing set shown in Figure \ref{HPRY-fig-psd-replication-K3-b}, however, can be replicated successfully, as only one cluster must be forced at any step of the forcing process.
\end{Example}
%%%%%%%%%%%%%%%%%%

\begin{figure}[h]
\begin{center}
	\begin{subfigure}[b]{0.3\textwidth}
	\begin{center}
		\includegraphics[scale=1, angle=-90]{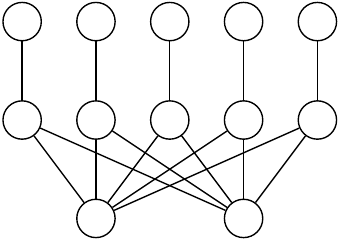} \qquad 
		\caption{Graph $G$}
		\label{HPRY-fig-psd-replication-a}
	\end{center}
	\end{subfigure}
	\begin{subfigure}[b]{0.3\textwidth}
	\begin{center}
		\includegraphics[scale=1, angle=-90]{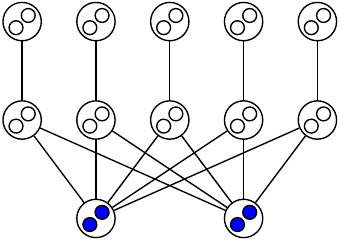} \qquad 
		\caption{$r = 2$}
		\label{HPRY-fig-psd-replication-b}
	\end{center}
	\end{subfigure}
	\quad
	\begin{subfigure}[b]{0.3\textwidth}
	\begin{center}
		\includegraphics[scale=1, angle=-90]{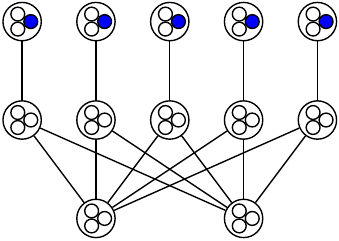} \qquad 
		\caption{$r = 3$}
		\label{HPRY-fig-psd-replication-c}
	\end{center}
	\end{subfigure}
\caption{Optimal AON $r$-fold PSD forcing sets}
\label{HPRY-fig-psd-replication}
\end{center}
\end{figure}

%%%%%%%%%%%%%%%%
\begin{Example} \label{HPRY-ex-not-AON-after-replication}
Suppose that we have the complete bipartite graph $K_{5,2}$ and let $G$ be the graph formed by attaching one leaf to each of the vertices in the partite set containing five vertices (Figure \ref{HPRY-fig-psd-replication-a}). Consider the (unique) optimal AON $r$-fold PSD forcing sets for $G$ shown in Figures \ref{HPRY-fig-psd-replication-b} and \ref{HPRY-fig-psd-replication-c}. When $r = 2$, the forcing set has two All clusters, so $\zr{2}^+(G) = 4$. When $r = 3$, the forcing set has five One clusters, so $\zr{3}^+(G) = 5$. Replicating either optimal forcing set onto the other blowup will generate a forcing set that is not minimum, hence not optimal.
\end{Example}
%%%%%%%%%%%%%%%%

We now prove further properties of AON $r$-fold PSD forcing sets and use these results to provide an alternate definition of the fractional PSD forcing number.

%%%%%%%%%%%%%%%%
\begin{Lemma} \label{HPRY-lem-AON-num-alls-large-r-P}
Let $G$ be a graph on $n$ vertices and fix $r \geq n$. Let $B$ be an optimal AON $r$-fold PSD forcing set for $G$ and let $B'$ be an AON $r$-fold PSD forcing set for $G$. Then $a(B) \leq a(B')$.
\end{Lemma}
\begin{proof}
Assume first that $\ell(B') < n$. Since $B$ is optimal, it is minimum, so $r \cdot a(B) + \ell(B) = |B| \leq |B'| = r \cdot a(B') + \ell(B')$. Dividing through by $r$ and manipulating this inequality yields
\[ a(B) - a(B') \leq \frac{\ell(B') - \ell(B)}{r} < \frac{n}{r} \leq 1. \]
Since $a(B) - a(B')$ is an integer, we must have $a(B) - a(B') \leq 0$, which proves the claim when $\ell(B') < n$. Now suppose that $\ell(B') = n$, so $a(B') = 0$. Since $r \geq n$, at most one force happens at each step, so we can replace the first cluster forced with a None cluster to obtain a new AON $r$-fold PSD forcing set $B''$ with $a(B'') = a(B') = 0$ and $\ell(B'') = n-1 < n$.
\end{proof}
%%%%%%%%%%%%%%%%

%%%%%%%%%%%%%%%%
\begin{Corollary} \label{HPRY-cor-AON-min-same-A-P}
Let $G$ be a graph on $n$ vertices and fix $r \geq n$. If $B$ and $B'$ are optimal AON $r$-fold PSD forcing sets for $G$, then $a(B) = a(B')$.
\end{Corollary}
%%%%%%%%%%%%%%%%

Thus for a fixed ``large enough" $r$, every optimal AON $r$-fold PSD forcing set for $G$ must contain the same number of All clusters (and, consequently, One clusters). Of particular interest is the case $r = n = |G|$. We define $\astar^+(G)$ to be the unique number of All clusters created in $\blowup{G}{n}$ by any optimal AON $n$-fold PSD forcing set for $G$, and define $\ellstar^+(G)$ to be the (unique) number of One clusters created in this manner.

%%%%%%%%%%%%%%%%
\begin{Proposition} \label{HPRY-prop-astar-homogeneity-P}
Let $G$ be a graph on $n$ vertices. For all $r \geq n$, if $B$ is an optimal AON $r$-fold PSD forcing set for $G$, then $a(B) = \astar^+(G)$.
\end{Proposition}
\begin{proof}
Let $\widetilde{B}$ be the AON $n$-fold PSD forcing set formed by replicating $B$ onto $\blowup{G}{n}$. By Lemma \ref{HPRY-lem-AON-num-alls-large-r-P}, $\astar^+(G) \leq a(\widetilde{B}) = a(B)$. Similarly, let $B'$ be the AON $r$-fold PSD forcing set formed by replicating any optimal AON $n$-fold PSD forcing set onto $\blowup{G}{r}$. By Lemma \ref{HPRY-lem-AON-num-alls-large-r-P}, $a(B) \leq a(B') = \astar^+(G)$, and thus equality holds.
\end{proof}
%%%%%%%%%%%%%%%%

%%%%%%%%%%%%%%%%
\begin{Corollary} \label{HPRY-cor-characterize-zr-large-r-P}
Let $G$ be a graph on $n$ vertices. For all $r \geq n$, $\zr{r}^+(G) = r \cdot \astar^+(G) + \ellstar^+(G)$. Additionally,
\[ \lim_{r \to \infty} \frac{\zr{r}^+(G)}{r} = \astar^+(G). \]
\end{Corollary}
%%%%%%%%%%%%%%%%

Before we can prove the final result of this section, which ties the fractional positive semidefinite forcing number into the machinery just developed, we require one final utility result.

%%%%%%%%%%%%%%%%
\begin{Lemma} \label{HPRY-lem-rfold-AON-lowerbound}
Let $G$ be a graph on $n$ vertices and choose $r \geq 2$. Then for any optimal AON $r$-fold PSD forcing set $B$, $\frac{|B|}{r} \geq \astar^+(G)$.
\end{Lemma}
\begin{proof}
First, suppose that $2 \leq r < n$. Let $\widetilde{B}$ be the AON $n$-fold PSD forcing set obtained by replicating $B$ onto $\blowup{G}{n}$. Then $a(B) = a(\widetilde{B})$ and $\ell(B) = \ell(\widetilde{B})$, so
\[ \frac{|B|}{r} = a(B) + \frac{\ell(B)}{r} = a(\widetilde{B}) + \frac{\ell(\widetilde{B})}{r} \geq a(\widetilde{B}) + \frac{\ell(\widetilde{B})}{n} = \frac{|\widetilde{B}|}{n} \, . \]

Let $B'$ be any optimal AON $n$-fold PSD forcing set for $G$. Since $B'$ is optimal, it is minimum, hence $|\widetilde{B}| \geq |B'|$. Therefore,
\[ \frac{|B|}{r}	\geq	\frac{|\widetilde{B}|}{n}	\geq	\frac{|B'|}{n} =	\astar^+(G) + \frac{\ellstar^+(G)}{n}	\geq	\astar^+(G), \]
which proves the claim when $r < n$.

If $r \geq n$, then Proposition \ref{HPRY-prop-astar-homogeneity-P} shows that $|B| = r \cdot \astar^+(G) + \ellstar^+(G)$ and the conclusion follows.
\end{proof}
%%%%%%%%%%%%%%%%

We conclude this section with an alternate characterization of fractional positive semidefinite forcing number.

%%%%%%%%%%%%%%%
\begin{Theorem} \label{HPRY-thm-characterize-zf-P}
For every graph $G$,
\[ Z_f^+(G)  = \astar^+(G). \]
\end{Theorem}
\begin{proof}
Recall that $Z_f^+ = \inf_{r \geq 2} \left\{ \frac{\zr{r}^+(G)}{r} \right\}$. By Corollary \ref{HPRY-cor-characterize-zr-large-r-P}, $Z_f^+(G) \leq \astar^+(G)$. Let $B$ be an optimal AON $r$-fold PSD forcing set for $G$ for some $r \geq 2$. Then by Corollary \ref{HPRY-cor-zr-alt-def} and Lemma \ref{HPRY-lem-rfold-AON-lowerbound}, $\frac{\zr{r}^+(G)}{r} = \frac{|B|}{r} \geq \astar^+(G)$, and thus equality holds.
\end{proof}
%%%%%%%%%%%%%%%

This shows that the fractional positive semidefinite forcing number of a graph is always a nonnegative integer, an interesting result in light of its fractional construction.

%%%%%%%%%%%%%%%%%%%%%%%%%%%%%%%%%%%%%%%%
\subsection{Three-color interpretation of fractional positive semidefinite forcing} \label{HPRY-sec-3color-PSD}

Motivated by the AON interpretation of the $r$-fold positive semidefinite forcing game, we consider a three-color forcing game that allows us to compute the fractional positive semidefinite forcing number for any graph without playing the $r$-fold game.

Let $G$ be a graph and consider the following \defbf{fractional positive semidefinite forcing game}, which is a three-color forcing game that uses the colors dark blue (target), light blue, and white. Assign to each vertex of $G$ one of these colors and let $\mathcal{B} = (\mathcal{D}, \mathcal{L})$, where $\mathcal{D}$ denotes the set of dark blue vertices and $\mathcal{L}$ denotes the set of light blue vertices.%
\footnote{Recall from Section \ref{HPRY-sec-intro-defs} that this is equivalent to writing $\mathcal{B} = \mathcal{D} \dcup \mathcal{L}$.}
We repeatedly apply the following \defbf{fractional positive semidefinite forcing rule}:

%%%%%%%%%%%%%
\begin{Definition}[fractional positive semidefinite forcing rule]
Let $\mathcal{B}_t = (\mathcal{D}_t, \mathcal{L}_t)$ denote the set of colored vertices of a graph $G$ at some step of the fractional positive semidefinite forcing process and let $W_1, \ldots, W_h$ denote the sets of vertices of the connected components of $G - \mathcal{D}_t$. If $u \in \left( \mathcal{D}_t \dcup (\mathcal{L}_t \cap W_i) \right)$ and $w \in W_i$ is the only light blue or white neighbor of $u$ in $G[\mathcal{D}_t \cup W_i]$, then $u$ can force $w$, i.e., $w$ can be colored dark blue.
\end{Definition}
%%%%%%%%%%%%%

Loosely speaking, we apply the disconnect rule from positive semidefinite zero forcing using the dark blue vertices of $G$, and then in each augmented component any dark or light blue vertex can force its only light blue or white neighbor. As usual, the goal of this forcing game is to choose the initial set $\mathcal{B}$ in such a way that by repeated application of this rule the entire graph can be forced (i.e., turned dark blue). If $G$ can be forced, then we say that the initial set $\mathcal{B}$ is a \defbf{fractional positive semidefinite (PSD) forcing set} for $G$. The \defbf{(three-color) fractional positive semidefinite forcing number} of $G$, denoted $\hat{Z}_f^+(G)$, is then defined as
\[ \hat{Z}_f^+(G) = \min\left\{ |\mathcal{D}| : (\mathcal{D}, \mathcal{L}) \text{ is a fractional PSD forcing set for $G$, for some $\mathcal{L}$} \right\}. \]
We say that a fractional PSD forcing set $\mathcal{B} = (\mathcal{D}, \mathcal{L})$ for $G$ is \defbf{optimal} if $|\mathcal{D}| = \hat{Z}_f^+(G)$ and no fractional PSD forcing set for $G$ with $|\mathcal{D}| = \hat{Z}_f^+(G)$ has fewer than $|\mathcal{L}|$ light blue vertices. We use $\ellstarhat^+(G)$ to denote the number of light blue vertices in any optimal fractional PSD forcing set for $G$, i.e., $\ellstarhat^+(G) = |\mathcal{L}|$.

The process of backforcing described for the $r$-fold positive semidefinite forcing game applies to the fractional positive semidefinite forcing game, albeit with a three-color modification. After a light blue vertex $u$ performs a force, all of its neighbors must necessarily be dark blue, and so we can backforce $u$ at the next forcing step.

The observant reader will notice that we have defined ``fractional positive semidefinite forcing number" twice: here, and in Section \ref{HPRY-sec-rfold-PSD-intro}. The final result of this section shows that this is not an error: the parameter $Z_f^+$, defined via an $r$-fold two-color game, is equal to the parameter $\hat{Z}_f^+$, defined via a three-color game.

%%%%%%%%%%%%%
\begin{Theorem} \label{HPRY-thm-zfp-equals-zhat}
For any graph $G$, $Z_f^+(G) = \hat{Z}_f^+(G)$.
\end{Theorem}
\begin{proof}
Let $|G| = n$ and let $B$ be an optimal AON $n$-fold PSD forcing set for $G$. By Theorem \ref{HPRY-thm-characterize-zf-P}, we have $a(B) = \astar^+(G) = Z_f^+(G)$. Color $\blowup{G}{n}$ with $B$ and color $G$ with $\widetilde{\mathcal{B}} = (\widetilde{\mathcal{D}}, \widetilde{\mathcal{L}})$, defined as follows: Let $\widetilde{\mathcal{D}} = \{ u : R_u \text{ is an All cluster in } \blowup{G}{n} \}$ and let $\widetilde{\mathcal{L}} = \{ u : R_u \text{ is a One cluster in } \blowup{G}{n} \}$. Since $B$ is an optimal AON $n$-fold PSD forcing set, exactly one cluster is forced at each step of the forcing process using $B$, and $\blowup{G}{n}$ can be forced. Further, backforcing is applied to One clusters in $\blowup{G}{n}$, and One clusters correspond to light blue vertices, to which backforcing can also be applied. Therefore, the forcing process used on $\blowup{G}{n}$ can be used to force $G$, so $\widetilde{\mathcal{B}}$ is a fractional PSD forcing set for $G$ and $\hat{Z}_f^+(G) \leq |\widetilde{\mathcal{D}}| = a(B) = Z_f^+(G)$.

Let $\mathcal{B} = (\mathcal{D}, \mathcal{L})$ be an optimal fractional PSD forcing set for $G$. The reverse inequality easily follows by associating elements of $\mathcal{D}$ with All clusters in $\blowup{G}{n}$ and elements of $\mathcal{L}$ with One clusters and applying Lemma \ref{HPRY-lem-AON-num-alls-large-r-P} and arguments similar to those above.
\end{proof}
%%%%%%%%%%%%%

%%%%%%%%%%%%%
\begin{Corollary}
For any graph $G$, $\ellstar^+(G) = \ellstarhat^+(G)$.
\end{Corollary}
%%%%%%%%%%%%%

As a consequence of these results, the $\hat{Z}_f^+$ and $\ellstarhat^+$ notations will be suppressed in favor of the simpler $Z_f^+$ and $\ellstar^+$.

In contrast to the process of computing the values of fractional versions of general graph parameters, computing the fractional positive semidefinite forcing number of a graph does not require any explicit knowledge of the $r$-fold analogue. If knowledge of $Z_f^+$ is all that is of interest, one can bypass the $r$-fold game and opt to play the fractional positive semidefinite forcing game instead. The benefit of taking a three-color approach is also demonstrated in Section \ref{HPRY-sec-skew}, where a three-color interpretation is used to obtain new results pertaining to skew zero forcing.

%%%%%%%%%%%%%%%%%%%%%%%%%%%%%%%%%%%%%%%%
\subsection{Results for fractional positive semidefinite forcing number} \label{HPRY-sec-zfp-results}

The fractional positive semidefinite forcing game allows us to easily prove many interesting properties of the fractional positive semidefinite forcing number.

%%%%%%%%%%%%%%%
\begin{Remark}
Any isolated vertex in a graph $G$ must be colored dark blue. Thus if $\delta(G) = 0$, then $Z_f^+(G) \geq \left| \{u \in V(G) : \operatorname{deg}(u) = 0 \} \right| \geq 1$.
\end{Remark}
%%%%%%%%%%%%%%%

%%%%%%%%%%%%%%%
\begin{Observation}
If a graph $G$ has connected components $\{G_i\}_{i=1}^m$, then $Z_f^+(G) = \sum_{i=1}^m Z_f^+(G_i)$ and $\ellstar^+(G) = \sum_{i=1}^m \ellstar^+(G_i)$.
\end{Observation}
%%%%%%%%%%%%%%%

Thus we are able to focus on connected graphs (as is customary for zero forcing).

%%%%%%%%%%%%%%%
\begin{Remark} \label{HPRY-rmk-frac-plus-ellstar-bds-P}
Let $G$ be a graph and let $\mathcal{B} = (\mathcal{D}, \mathcal{L})$ be a fractional PSD forcing set for $G$. The set $B = \mathcal{D} \dcup \mathcal{L}$ is a positive semidefinite zero forcing set for $G$, so $Z^+(G) \leq |B| = |\mathcal{D}| + |\mathcal{L}|$. If $\mathcal{B}$ is optimal, then this shows that $Z^+(G) \leq Z_f^+(G) + \ellstar^+(G)$.
\end{Remark}
%%%%%%%%%%%%%%%

A natural question in light of this remark is whether $Z^+(G) = Z_f^+(G) + \ellstar^+(G)$. By taking a minimum positive semidefinite zero forcing set for $G$ and changing some vertices to light blue, it may be possible to obtain an optimal fractional PSD forcing set for $G$. Even though this works for some graphs, the next example provides a graph for which this technique fails.

%%%%%%%%%%%%%%%
\begin{Example} \label{HPRY-ex-Kmn-ctrex}
Let $G$ be the graph shown in Example \ref{HPRY-ex-not-AON-after-replication}, a $K_{5,2}$ with one leaf appended to each vertex in the partite set on 5 vertices. By coloring each of the leaves light blue, we can force each of their neighbors, and using the disconnect rule we can subsequently backforce the leaves and force all of $G$. Thus $Z_f^+(G) = 0$ and $\ellstar^+(G) = 5$, but it is known that $Z^+(G) = 2 < 0 + 5 = Z_f^+(G) + \ellstar^+(G)$. The key to this example is that the set $B = \mathcal{L}$ is a minimal positive semidefinite zero forcing set for $G$, but it is not a minimum positive semidefinite zero forcing set.
\end{Example}
%%%%%%%%%%%%%%%

%%%%%%%%%%%%%%%
\begin{Remark} \label{HPRY-rmk-light-blue-P}
If $\mathcal{B} = (\mathcal{D}, \mathcal{L})$ is an optimal fractional PSD forcing set for a connected graph $G$, then any vertex that is colored light blue must perform a force before it is itself forced; if not, then that vertex can be colored white to obtain a fractional PSD forcing set with the same number of dark blue vertices and fewer light blue vertices, contradicting the optimality of $\mathcal{B}$. Additionally, no two light blue vertices in an optimal fractional PSD forcing set can be adjacent, as one would have to force the other before the other has performed a force. Therefore, $\mathcal{L}$ is an independent set in $G$, so $\ellstar^+(G) \leq \alpha(G)$.
\end{Remark}
%%%%%%%%%%%%%%%

The following result pertains to the (two-color) positive semidefinite zero forcing game.

%%%%%%%%%%%%%%%
\begin{Lemma}[\cite{HPRY-travisThesis}, Lemma 2.1.1] \label{HPRY-lemma-travis-swap}
Let $G$ be a graph and let $B$ be a positive semidefinite zero forcing set of $G$. If $v \in B$ is the vertex that performs the first force, $v \force w$, where $w$ is a white neighbor of $v$, then $(B \setminus \{v\}) \cup \{w\}$ is a positive semidefinite zero forcing set of $G$.
\end{Lemma}
%%%%%%%%%%%%%%%

We now present a three-color version of Lemma \ref{HPRY-lemma-travis-swap}. The proof is similar to the proof of the two-color version found in \cite{HPRY-travisThesis} and is omitted.

%%%%%%%%%%%%%%%
\begin{Lemma} \label{HPRY-lemma-travis-3color}
Let $G$ be a graph and let $\mathcal{B} = (\mathcal{D}, \mathcal{L})$ be a fractional PSD forcing set for $G$. Suppose that the first force, $v \force w$, is performed by some $v \in \mathcal{D}$ on some $w \notin \mathcal{L}$. Let $\widetilde{\mathcal{D}} = \left( \mathcal{D} \setminus \{v\} \right) \cup \{w\}$. Then $\widetilde{\mathcal{B}} = (\widetilde{\mathcal{D}}, \mathcal{L})$ is also a fractional PSD forcing set for $G$.
\end{Lemma}
%%%%%%%%%%%%%%%

%%%%%%%%%%%%%%%
\begin{Theorem} \label{HPRY-thm-first-force-light-blue-P}
If $G$ is a graph with at least one edge, then $G$ has an optimal fractional PSD forcing set with which the first force can be performed by a light blue vertex.
\end{Theorem}
\begin{proof}
Suppose for the sake of contradiction that $G$ does not have an optimal fractional PSD forcing set with which the first force can be performed by a light blue vertex. Note that if the first force with an optimal set can be done without using the disconnect rule, then this force must be done by a light blue vertex (else the set is not optimal), so our assumption implies that the disconnect rule must be applied to perform the first force with any optimal fractional PSD forcing set. Let $\mathcal{B} = (\mathcal{D}, \mathcal{L})$ be an optimal fractional PSD forcing set such that $|W_1|$ is minimum, where $W_1, W_2, \ldots, W_h$ are the sets of vertices of the connected components of $G - \mathcal{D}$ and $|W_1| \leq |W_2| \leq \cdots \leq |W_h|$. As noted in Remark 1.14 of \cite{HPRY-nathanThesis}, we can assume that the first vertex forced lies in $W_1$. Let $v \force w$ denote the first force, where $v \in \mathcal{D}$ and $w \in W_1$.

By Lemma \ref{HPRY-lemma-travis-3color}, the set $\widetilde{\mathcal{B}} = (\widetilde{\mathcal{D}}, \mathcal{L})$ with $\widetilde{\mathcal{D}} = \left( \mathcal{D} \setminus \{v\} \right) \cup \{w\}$ is also an optimal fractional PSD forcing set for $G$. Since $w$ must be the only non-dark-blue neighbor of $v$ in $W_1$, it must be the case that $v$ joins a component other than $W_1$ in $G - \widetilde{\mathcal{D}}$; further, in $G - \widetilde{\mathcal{D}}$, the component $W_1$ will not contain the vertex $w$, and may split into multiple smaller components. If $W_1 \neq \{w\}$, then this argument shows that there must be a component with fewer than $|W_1|$ vertices in $G - \widetilde{\mathcal{D}}$, which contradicts the choice of $\mathcal{B}$; thus we must have $W_1 = \{w\}$. However, the first force in $G$ using $\widetilde{\mathcal{B}}$ can therefore be chosen as $w \force v$, which can be done without applying the disconnect rule; by the comments above, $w$ can thus be light blue, contradicting optimality of $\mathcal{B}$. Therefore, $G$ must have an optimal fractional PSD forcing set with which the first force can be performed by a light blue vertex.
\end{proof}
%%%%%%%%%%%%%%%

Theorem \ref{HPRY-thm-first-force-light-blue-P} yields a lower bound on $Z_f^+(G)$ as a corollary.

%%%%%%%%%%%%%%%
\begin{Corollary} \label{HPRY-cor-zfp-geq-deltaM1}
For any graph $G$, $\delta(G) - 1 \leq Z_f^+(G)$.
\end{Corollary}
\begin{proof}
The result is trivial for $\delta(G) \leq 1$. If $\delta(G) \geq 2$, then $G$ has an edge, so by Theorem \ref{HPRY-thm-first-force-light-blue-P} there exists some optimal fractional PSD forcing set $\mathcal{B} = (\mathcal{D}, \mathcal{L})$ such that the first force in $G$ can be done by some $u \in \mathcal{L}$. Remark \ref{HPRY-rmk-light-blue-P} asserts that $u$ has no light blue neighbors, and all white neighbors of $u$ must be in the same component of $G - \mathcal{D}$. Since $u$ can force, all but one of its neighbors must be dark blue. Thus $|\mathcal{D}| \geq |N(u)| - 1 \geq \delta(G) - 1$.
\end{proof}
%%%%%%%%%%%%%%%

An additional corollary to Theorem \ref{HPRY-thm-first-force-light-blue-P} gives a lower bound on $\ellstar^+(G)$ in the case where $G$ has at least one edge.

%%%%%%%%%%%%%%%
\begin{Corollary} \label{HPRY-cor-ellstar-with-edge-P}
If $G$ is a graph with at least one edge, then $\ellstar^+(G) \geq 1$.
\end{Corollary}
%%%%%%%%%%%%%%%

The following result is a two-color analogue of Theorem \ref{HPRY-thm-first-force-light-blue-P} that applies to the positive semidefinite zero forcing game. The proof is similar to that of Theorem \ref{HPRY-thm-first-force-light-blue-P} and is omitted.

%%%%%%%%%%%%%%%
\begin{Theorem} \label{HPRY-thm-zfp-pivoting}
If $G$ is a graph with at least one edge, then there exists a minimum positive semidefinite zero forcing set for $G$ such the first force can be done without using the disconnect rule.
\end{Theorem}
%%%%%%%%%%%%%%%

With Theorem \ref{HPRY-thm-zfp-pivoting}, we can obtain an improved upper bound on $Z_f^+(G)$.

%%%%%%%%%%%%%%%
\begin{Corollary} \label{HPRY-cor-zfp-leq-zpminus1}
For any graph $G$ with at least one edge, $Z_f^+(G) \leq Z^+(G) - 1$.
\end{Corollary}
\begin{proof}
Theorem \ref{HPRY-thm-zfp-pivoting} ensures that there is some minimum positive semidefinite zero forcing set $B$ such that the first force using $B$ can be done without using the disconnect rule. If $\mathcal{B}$ is obtained by coloring the vertex that performs this first force light blue and all of the other vertices in $B$ dark blue, then $\mathcal{B}$ is a fractional PSD forcing set with $Z^+(G) - 1$ dark blue vertices.
\end{proof}
%%%%%%%%%%%%%%%

%%%%%%%%%%%%%%%%%%%%%%%%%%%%%%%%%%%%%%%
\subsection{Fractional positive semidefinite forcing numbers for graph families}

In this section, we determine the fractional PSD forcing numbers for certain graph families, illustrating the utility of some of the results in Section \ref{HPRY-sec-zfp-results}.

%%%%%%%%%%%
\begin{Example}
Let $n \geq 2$ and let $V(K_n) = \{v_1, v_2, \ldots, v_n\}$. Note that $Z^+(K_n) = n-1$ \cite[Example 46.4.2]{HPRY-HLA2ch46}. Applying Corollaries \ref{HPRY-cor-zfp-geq-deltaM1} and \ref{HPRY-cor-zfp-leq-zpminus1}, $n-2 = \delta(K_n) - 1 \leq Z_f^+(K_n) \leq Z^+(K_n) - 1 = n-2$ and thus equality holds. By Corollary \ref{HPRY-cor-ellstar-with-edge-P}, $\ellstar^+(K_n) \geq 1$. The set $\mathcal{B} = (\{v_1, v_2, \ldots, v_{n-2}\}, \{v_{n-1}\})$ is an optimal fractional PSD forcing set for $K_n$, so $Z_f^+(K_n) = n-2$ and $\ellstar^+(K_n) = 1$.
\end{Example}
%%%%%%%%%%%

In each of the next four examples, optimality of the exhibited fractional PSD forcing sets is obtained by application of Corollaries \ref{HPRY-cor-zfp-geq-deltaM1} and \ref{HPRY-cor-ellstar-with-edge-P}.

%%%%%%%%%%%
\begin{Example}
For any $n \geq 2$, the set $\mathcal{B} = (\emptyset, \{v_1\})$ is an optimal fractional PSD forcing set for $P_n$, where $V(P_n) = \{v_1, v_2, \ldots, v_n\}$ in path order, so $Z_f^+(P_n) = 0$ and $\ellstar^+(P_n) = 1$.
\end{Example}
%%%%%%%%%%%

%%%%%%%%%%%
\begin{Example}
For any $n \geq 3$, the set $\mathcal{B} = (\{v_1\}, \{v_2\})$ is an optimal fractional PSD forcing set for $C_n$, where $V(C_n) = \{v_1, v_2, \ldots, v_n\}$ in cycle order, so $Z_f^+(C_n) = 1$ and $\ellstar^+(C_n) = 1$.
\end{Example}
%%%%%%%%%%%

%%%%%%%%%%%
\begin{Example}
Let $n \geq 4$ and consider the wheel on $n$ vertices, $W_n$, which is obtained by adding a vertex $w$ adjacent to every vertex of $C_{n-1}$. If $\mathcal{B} = (\mathcal{D},\mathcal{L})$ is any optimal fractional PSD forcing set for $C_{n-1}$, then $\widetilde{\mathcal{B}} = (\mathcal{D} \cup \{w\}, \mathcal{L})$ is an optimal fractional PSD forcing set for $W_n$, so $Z_f^+(W_n) = 2$ and $\ellstar^+(W_n) = 1$.
\end{Example}
%%%%%%%%%%%

%%%%%%%%%%%
\begin{Example}
Let $p \geq q \geq 1$ and consider $K_{p,q}$, the complete bipartite graph on partite sets $P$ and $Q$ with $|P| = p$ and $|Q| = q$. Let $\mathcal{D}$ be a set containing any $(q-1)$ elements of $Q$ and let $\mathcal{L}$ be a set containing any one element of $P$; then $\mathcal{B} = (\mathcal{D}, \mathcal{L})$ is an optimal fractional PSD forcing set for $K_{p,q}$, so $Z_f^+(K_{p,q}) = q-1$ and $\ellstar^+(K_{p,q}) = 1$.
\end{Example}
%%%%%%%%%%%

As a final example, we consider the fractional PSD forcing number of a tree.

%%%%%%%%%%%
\begin{Example}
Suppose that $T$ is a tree of order at least 2. We have $Z^+(T) = 1$ \cite[Example 46.4.3]{HPRY-HLA2ch46}, so Corollary \ref{HPRY-cor-zfp-leq-zpminus1} implies that $0 \leq Z_f^+(T) \leq Z^+(T) - 1 = 0$ and hence equality holds. If we let $\mathcal{L}$ be any leaf of $T$, then $\mathcal{B} = (\emptyset, \mathcal{L})$ is an optimal fractional PSD forcing set, so $Z_f^+(T) = 0$ and $\ellstar^+(T) = 1$.
\end{Example}
%%%%%%%%%%%

%%%%%%%
% !TeX root = ./frac-zf-master.tex

%%%%%%%%%%%%%%%%%%%%%%%%%%%%%%%%%%%%%%%%%%%%%%%%%%%%%%
\section{Three-color interpretation of skew zero forcing} \label{HPRY-sec-skew}

In this section, we introduce a three-color interpretation of the skew zero forcing game and use this to show that the skew zero forcing number and ``fractional (zero) forcing number" of a graph are equal. Using the three-color interpretation, we derive new results pertaining to skew zero forcing number and the associated coloring process.

%%%%%%%%%%%%%%%%%%%%%%%%%%%%%
\subsection{The three-color skew zero forcing game} \label{HPRY-sec-skew-ZF}

Consider the following three-color forcing game played on a graph $G$. Choose an initial set of dark blue vertices, $\mathcal{D}$, and a set of light blue vertices, $\mathcal{L}$, and let $\mathcal{B} = (\mathcal{D}, \mathcal{L})$; color all other vertices of $G$ white. The forcing rule is as follows: 
\begin{Definition}[three-color skew zero forcing rule]
If $w$ is the only non-dark-blue neighbor of a dark blue or light blue vertex $u$, then $u$ can force $w$.
\end{Definition}

The set $\mathcal{B}$ is a \defbf{three-color skew zero forcing set} if $G$ can be forced after repeated application of the three-color skew zero forcing rule. We define
\[ \hat{Z}^-(G) = \min\left\{ |\mathcal{D}| : (\mathcal{D}, \mathcal{L}) \text{ is a three-color skew zero forcing set for $G$, for some $\mathcal{L}$} \right\}. \]
A three-color skew zero forcing set $\mathcal{B} = (\mathcal{D}, \mathcal{L})$ is \defbf{optimal} if $|\mathcal{D}| = \hat{Z}^-(G)$ and no such forcing set for $G$ has fewer light blue vertices than $\mathcal{B}$. Let $\ellstar^-(G)$ denote the number of light blue vertices in any optimal three-color skew zero forcing set for $G$, i.e., $\ellstar^-(G) = |\mathcal{L}|$.

The inclusion of the word ``skew" in the development of $\hat{Z}^-(G)$ is not an accident. It is easy to see that the three-color skew zero forcing game is equivalent to the (two-color) skew zero forcing game described in Section \ref{HPRY-sec-intro-ZF}: dark blue vertices correspond to (regular) blue vertices in two-color skew zero forcing, light blue vertices correspond to white vertices that perform white vertex forcing, and white vertices that do not perform a white vertex force are the same in both cases. Therefore, $\hat{Z}^-(G) = Z^-(G)$, and we are free to use the more familiar notation $Z^-(G)$ when discussing the three-color game.

%%%%%%%%%%%%%%%%
\begin{Remark} \label{HPRY-rmk-zfp-leq-zm}
Notice that any three-color skew zero forcing set for a graph $G$ is also a fractional PSD forcing set for $G$: playing the three-color skew zero forcing game is equivalent to playing the fractional PSD zero forcing game without using the disconnect rule. Therefore, $Z_f^+(G) \leq Z^-(G)$.
\end{Remark}
%%%%%%%%%%%%%%%%

From this point forward, since they give more information than their two-color counterparts, we will focus on three-color skew zero forcing sets, and usually omit the ``three-color" descriptor for the sake of brevity.

%%%%%%%%%%%%%%%%%%%%%%%%%%%%%
\subsection{General results for skew zero forcing}

The three-color interpretation easily lends itself to making observations about skew zero forcing number of a graph. The next two results are well-known for $Z^-(G)$ using the two-color approach, where we interpret $\ellstar^-(G)$ as the number of vertices that perform white vertex forces in that case.

%%%%%%%%%
\begin{Remark}
Any isolated vertex in a graph $G$ must be colored dark blue, so if $\delta(G) = 0$, then $Z^-(G) \geq \left| \{ u \in V(G) : \operatorname{deg}(u) = 0 \} \right| \geq 1$.
\end{Remark}
%%%%%%%%%

%%%%%%%%%
\begin{Observation}
If a graph $G$ has connected components $\{G_i\}_{i=1}^m$, then $Z^-(G) = \sum_{i=1}^m Z^-(G_i)$ and $\ellstar^-(G) = \sum_{i=1}^m \ellstar^-(G_i)$.
\end{Observation}
%%%%%%%%%

As is customary, we are able to focus our attention on connected graphs.

%%%%%%%%%
\begin{Remark}
For every connected graph $G$, $\delta(G) - 1 \leq Z^-(G)$. This is because if a candidate skew zero forcing set does not contain at least $\delta(G)-1$ dark blue vertices, then every dark blue or light blue vertex has at least two white or light blue neighbors, so the forcing process cannot start.
\end{Remark}
%%%%%%%%%

%%%%%%%%%
\begin{Remark}
Suppose that $G$ is a connected graph on 2 or more vertices and color each of its vertices dark blue. Any one adjacent pair can then be re-colored white and light blue (in either order), so $Z^-(G) \leq |G| - 2$.
\end{Remark}
%%%%%%%%%

%%%%%%%%%
\begin{Remark} \label{HPRY-rmk-skew-and-reg-zf-bds}
For every connected graph $G$, we have $Z^-(G) \leq Z(G) \leq Z^-(G) + \ellstar^-(G)$. The first inequality follows because every zero forcing set for a graph $G$ is also a skew zero forcing set for $G$. For the second, note that if $\mathcal{B} = (\mathcal{D}, \mathcal{L})$ is an optimal skew zero forcing set, then $B = \mathcal{D} \dcup \mathcal{L}$ is a (standard) zero forcing set.
\end{Remark}
%%%%%%%%%

The justification for the next observation is the same as that given in Remark \ref{HPRY-rmk-light-blue-P}.

%%%%%%%%%
\begin{Observation}
If $\mathcal{B} = (\mathcal{D}, \mathcal{L})$ is an optimal skew zero forcing set for a connected graph $G$, then any vertex that is colored light blue must perform a force before it is itself forced. No two light blue vertices in an optimal skew zero forcing set can be adjacent. The set $\mathcal{L}$ is an independent set in $G$, and $\ellstar^-(G) \leq \alpha(G)$.
\end{Observation}
%%%%%%%%%

%%%%%%%%%
\begin{Remark} \label{HPRY-rmk-ellstarm-bounds}
For a graph $G$, the quantity $|G| - Z^-(G)$ is the number of non-dark-blue vertices in an optimal skew zero forcing set. In the worst case, half of these vertices would need to be colored light blue to force their white neighbors, so $0 \leq \ellstar^-(G) \leq \left\lfloor \frac{|G| - Z^-(G)}{2} \right\rfloor \leq \left\lfloor \frac{|G|}{2} \right\rfloor < |G|$.
\end{Remark}
%%%%%%%%%

%%%%%%%%%%%%%%%%%%%%%%%%%%%%%
\subsection{Skew zero forcing as fractional zero forcing}

In this section, we develop an $r$-fold version of the standard zero forcing game and use it to prove that the ``fractional (zero) forcing number" of a graph is equal to the skew zero forcing number of the graph. This treatment is similar to the positive semidefinite case discussed in Sections \ref{HPRY-sec-rfold-PSD-intro} and \ref{HPRY-sec-global-PSD}.

Let $G$ be a graph and for some $r \in \mathbb{N}$ consider the following \defbf{$r$-fold forcing game}, which is a two-color forcing game played on $\blowup{G}{r}$, the $r$-blowup of $G$. As in any zero forcing game, we initially color some set $B \subseteq V(\blowup{G}{r})$ blue and then try to force $\blowup{G}{r}$ through repeated application of the following \defbf{$r$-fold forcing rule}:

%%%%%%%%%%%%%%
\begin{Definition}[$r$-fold forcing rule]
At some step $t$ of the forcing process, let $B_t$ denote the set of blue vertices in $\blowup{G}{r}$. If $u \in B_t$ and $|N(u) \setminus B_t| \leq r$, then $u$ can force $N(u) \setminus B_t$, i.e., all white neighbors of $u$ can be colored blue simultaneously.
\end{Definition}
%%%%%%%%%%%%%%

The $r$-fold forcing rule is exactly the $r$-forcing rule found in \cite{HPRY-davila-kforcing}, although applied to $\blowup{G}{r}$ instead of $G$. The $r$-fold forcing game was developed in the spirit of fractional graph theory \cite{HPRY-fgt}, while the $r$-forcing process described in \cite{HPRY-davila-kforcing} is more general. We have chosen to use different terminology with our treatment to emphasize this key difference.

If $\blowup{G}{r}$ can be forced, then the initial set of blue vertices is called an \defbf{$r$-fold forcing set} for $G$. A \defbf{minimum $r$-fold forcing set} is an $r$-fold forcing set of minimum cardinality. The \defbf{$r$-fold forcing number} of $G$, $\zr{r}(G)$, is the cardinality of a minimum $r$-fold forcing set.%
\footnote{Note that $\zr{r}(G) = F_r(\blowup{G}{r})$, where $F_k(H)$ is the $k$-forcing number of a graph $H$; see \cite{HPRY-davila-kforcing}.}
We define the \defbf{fractional forcing number} of $G$ as
\[ Z_f(G) = \inf_{r \in \mathbb{N}} \left\{ \frac{\zr{r}(G)}{r}  \right\}. \]

Clearly, $\zr{1}(G) = Z(G)$. By an argument similar to that used in Section \ref{HPRY-sec-rfold-PSD-intro}, it is easy to see that $\zr{r}(G) \leq r \cdot Z(G)$ for $r \geq 2$, so we can equivalently define fractional forcing number as
\[ Z_f(G) = \inf_{r \geq 2} \left\{ \frac{\zr{r}(G)}{r} \right\}. \]

Our goal in this section is to prove that $Z_f(G) = Z^-(G)$ for any graph $G$. In order to do this, we will follow an approach similar to that used in Section \ref{HPRY-sec-global-PSD}, with the noted difference that we have a three-color interpretation of skew zero forcing that can be used to simplify some of our arguments.

The global view of the $r$-fold forcing game, analogous to that of the $r$-fold positive semidefinite forcing game, will also be considered. Since backforcing does not apply to this game, in addition to All, One, and None clusters in $\blowup{G}{r}$, we consider one other type of cluster: a \defbf{Most cluster} is a cluster in which all but one vertex is colored blue. We consider Most clusters only for $r \geq 3$, as when $r = 2$ a Most cluster is equivalent to a One cluster. An \defbf{All-Most-One-None (AMON) $r$-fold forcing set} is an $r$-fold forcing set for $G$ that creates All, Most, One, and None clusters in $\blowup{G}{r}$. As before, we let $a(B)$ denote the number of All clusters and $\ell(B)$ denote the number of One clusters created in $\blowup{G}{r}$ by an AMON $r$-fold forcing set $B$; we introduce $m(B)$ to denote the number of Most clusters created by $B$. If $B$ is an AMON $r$-fold forcing set, then $|B| = r \cdot a(B) + (r-1) \cdot m(B) + \ell(B) = r \left( a(B) + m(B) \right) + \ell(B) - m(B)$.

Many of the remarks and observations from Section \ref{HPRY-sec-global-PSD} apply to the global interpretation of the $r$-fold forcing game and we omit or reduce their proofs. As before, we note that forcing into a cluster $R_u$ is equivalent to forcing $R_u$ and a cluster that is forced becomes an All cluster. Each cluster performs at most one force.

%%%%%%%%%%%%%%%%%
\begin{Theorem} \label{HPRY-thm-AMON-S}
For any graph $G$ and any $r \geq 2$, an AMON minimum $r$-fold forcing set for $G$ exists, as does a forcing process in which at each step either exactly one cluster is forced or a One cluster and a Most cluster (or, when $r = 2$, two One clusters) are forced simultaneously.
\end{Theorem}
\begin{proof}
The result is trivially true for $r = 2$, so assume that $r \geq 3$. Let $B$ be a minimum $r$-fold forcing set for $G$ and suppose that $B$ is not AMON. Create a chronological list of forces in $\blowup{G}{r}$ and suppose that at step $\ell \geq 1$ we have $x \force R_u$ for some $u$, and $R_u$ is the only cluster forced at this step. If $R_u$ is not a One or a None cluster, then consider the set $B'$ obtained by replacing $R_u$ with a One cluster. Since $B'$ is a forcing set with fewer blue vertices than $B$, this contradicts that $B$ is minimum. Thus if a single cluster is forced at some step of the forcing process, then it is either a One or a None cluster.

Now, suppose that at step $\ell \geq 1$ we have $x \force W \subseteq \left( R_{u_1} \cup R_{u_2} \cup \cdots \cup R_{u_m} \right)$ for some $m \geq 2$, where each $R_{u_j}$ contains at least one white vertex. Since $x$ is performing a force, it has at most $r$ white neighbors. Thus we can perform a partial consolidation on the blue vertices spread among the $R_{u_j}$ as follows: Convert $R_{u_1}, R_{u_2}, \ldots, R_{u_{m-2}}$ into All clusters, convert $R_{u_{m-1}}$ into a Most cluster, and leave the remaining blue vertices in $R_{u_m}$. If we let $\widetilde{B}$ be the set obtained by performing this particular partial consolidation on $B$, then $\widetilde{B}$ is also a minimum $r$-fold forcing set for $G$. Notice that after partial consolidation, minimality of $B$ implies that $R_{u_m}$ must be a One cluster. Therefore, after partial consolidation, $x$ will force exactly two clusters, simultaneously -- a Most cluster and a One cluster.

By performing partial consolidation, each cluster will become an All, Most, One, or None cluster, and a forcing process exists with which at each step either a single One or None cluster will be forced, or a Most and a One cluster will be forced simultaneously.
\end{proof}
%%%%%%%%%%%%%%%%%

The type of AMON minimum $r$-fold forcing set guaranteed by Theorem \ref{HPRY-thm-AMON-S} is called an \defbf{optimal AMON $r$-fold forcing set} for $G$. We emphasize that optimal AMON forcing sets are minimum, so $\zr{r}(G)$ is the size of such a set, and there is a corresponding forcing process in which at most two clusters are forced simultaneously. Using such a set and the associated forcing process, $\blowup{G}{r}$ will have a global AMON structure at each forcing step.

%%%%%%%%%%%%%%%%%
\begin{Corollary} \label{HPRY-cor-ell-geq-m-S}
For every graph $G$ and $r \geq 2$, there exists an optimal AMON $r$-fold forcing set for $G$. If $B$ is any such set, then $\ell(B) \geq m(B)$.
\end{Corollary}
\begin{proof}
For each Most cluster in an optimal AMON $r$-fold forcing set there exists a corresponding One cluster that is forced simultaneously using the forcing process guaranteed by Theorem \ref{HPRY-thm-AMON-S}, so the number of Most clusters cannot exceed the number of One clusters.
\end{proof}
%%%%%%%%%%%%%%%%%

To obtain our main results of this section, we require a way to convert an AMON $r$-fold forcing set for $G$ into a (three-color) skew zero forcing set for $G$, and vice-versa.

%%%%%%%%%%%%%%%%
\begin{Remark} \label{HPRY-rmk-conversion}
For $r \geq 2$, let $B$ be an optimal AMON $r$-fold forcing set for a graph $G$. Color $\blowup{G}{r}$ with $B$ and let $\widetilde{\mathcal{B}} = (\widetilde{\mathcal{D}}, \widetilde{\mathcal{L}})$, where $\widetilde{\mathcal{D}} = \{ u : R_u \text{ is an All or Most cluster} \}$ and $\widetilde{\mathcal{L}} = \{ u : R_u \text{ is a One cluster}\}$. It is easy to see that $\widetilde{\mathcal{B}}$ is a skew zero forcing set for $G$. Similarly, let $\mathcal{B} = (\mathcal{D}, \mathcal{L})$ be a skew zero forcing set for $G$. Color $\blowup{G}{r}$ according to the following rule: If $u \in \mathcal{D}$, then make $R_u$ an All cluster, and if $u \in \mathcal{L}$, then make $R_u$ a One cluster. The set $\widetilde{B}$ of blue vertices is an AMON $r$-fold forcing set for $G$ (with $m(\widetilde{B}) = 0$).
\end{Remark}
%%%%%%%%%%%%%%%%

%%%%%%%%%%%%%%%%
\begin{Definition}
Regardless of whether we transform an $r$-fold forcing set into a three-color skew zero forcing set or a three-color skew zero forcing set into an $r$-fold forcing set, we call the process described in Remark \ref{HPRY-rmk-conversion} \defbf{conversion}.
\end{Definition}
%%%%%%%%%%%%%%%%

When performing conversion, we will always specify which type of set is being converted.

%%%%%%%%%%%%%%%%%
\begin{Proposition} \label{HPRY-prop-aplusm-equals-zm}
Let $G$ be a graph on $n$ vertices. If $r \geq n$ and $B$ is an optimal AMON $r$-fold forcing set, then $a(B) + m(B) = Z^-(G)$.
\end{Proposition}
\begin{proof}
Assume the hypotheses. Converting $B$ into a skew zero forcing set $\widetilde{\mathcal{B}} = (\widetilde{\mathcal{D}}, \widetilde{\mathcal{L}})$ yields $Z^-(G) \leq |\widetilde{\mathcal{D}}| = a(B) + m(B)$.

Now, let $\mathcal{B} = (\mathcal{D}, \mathcal{L})$ be an optimal skew zero forcing set for $G$ and convert $\mathcal{B}$ into an AMON $r$-fold forcing set $\widetilde{B}$. Since $B$ is optimal, it is minimum, so $|B| \leq |\widetilde{B}|$. Thus 
\[ a(B) + m(B) + \frac{\ell(B) - m(B)}{r} = \frac{|B|}{r} \leq \frac{|\widetilde{B}|}{r} = |\mathcal{D}| + \frac{|\mathcal{L}|}{r} = Z^-(G) + \frac{\ellstar^-(G)}{r} \, . \]
Since $0 \leq \ell(B) - m(B) \leq \ell(B)$ by Corollary \ref{HPRY-cor-ell-geq-m-S}, $\ellstar^-(G) < n$ by Remark \ref{HPRY-rmk-ellstarm-bounds}, and $n \leq r$ by assumption, applying the floor function through the above inequality yields $a(B) + m(B) \leq Z^-(G)$, provided that $\ell(B) < n$. This must be the case, because if $\ell(B) = n$ then $B$ cannot be minimum ($r \geq n$ implies that the first cluster forced could be a None).
\end{proof}
%%%%%%%%%%%%%%%%%

%%%%%%%%%%%%%%%%%
\begin{Corollary} \label{HPRY-cor-limit-zr-over-r}
If $G$ is a graph on $n$ vertices, then
\[ \lim_{r \to \infty} \frac{\zr{r}(G)}{r} = Z^-(G). \]
\end{Corollary}
%%%%%%%%%%%%%%%%%

%%%%%%%%%%%%%%%%%
\begin{Proposition} \label{HPRY-prop-BoverR-geq-zm}
For any $r \geq 2$ and any graph $G$, $\frac{\zr{r}(G)}{r} \geq Z^-(G)$.
\end{Proposition}
\begin{proof}
Let $B$ be an optimal AMON $r$-fold forcing set for $G$ and let $\mathcal{B} = (\mathcal{D}, \mathcal{L})$ be obtained by converting $B$ into a skew zero forcing set. By Corollary \ref{HPRY-cor-ell-geq-m-S}, $\ell(B) \geq m(B)$, so
\[ \frac{\zr{r}(G)}{r} = \frac{|B|}{r} = a(B) + m(B) + \frac{\ell(B) - m(B)}{r} \geq a(B) + m(B) = |\mathcal{D}| \geq Z^-(G). \qedhere \]
\end{proof}
%%%%%%%%%%%%%%%%%

%%%%%%%%%%%%%%%%%
\begin{Theorem} \label{HPRY-thm-zf-equals-skew}
For any graph $G$,
\[ Z_f(G) = Z^-(G). \]
\end{Theorem}
%%%%%%%%%%%%%%%%%

%%%%%%%%%%%%%%%%%%%%%%%%%%%%%
\subsection{Leaf-stripping and skew zero forcing number} \label{HPRY-sec-leaf-stripping}

In this section, we prove results about graphs with leaves and show that skew zero forcing number is unchanged by removing leaves and their neighbors. A leaf-stripping algorithm is presented and used to characterize graphs $G$ that have $Z^-(G) = 0$. For convenience, we define $Z^-(\emptyset) = 0$.

%%%%%%%%%
\begin{Lemma} \label{HPRY-lem-leaf-properties-S}
Let $G$ be a graph with leaf $u \in V(G)$ and let $v \in V(G)$ be the neighbor of $u$. Let $\mathcal{B} = (\mathcal{D}, \mathcal{L})$ be an optimal skew zero forcing set for $G$. i) If $u$ is either light or dark blue, then $v$ is white. ii) If $u$ is white, then $v$ is not dark blue.
\end{Lemma}
\begin{proof}
For the first claim, since $u \in \mathcal{B}$ and $v$ is the only neighbor of $u$, we can choose $u \force v$ as the first step in the forcing process. In this case $v$ must be white because $\mathcal{B}$ is optimal. For the second claim, if $v \in \mathcal{D}$, then the set $\widetilde{\mathcal{B}} = (\mathcal{D} \setminus \{v\}, \mathcal{L} \cup \{u\})$ has fewer dark blue vertices than $\mathcal{B}$ but is a skew zero forcing set for $G$, contradicting the optimality of $\mathcal{B}$.
\end{proof}
%%%%%%%%%

%%%%%%%%%
\begin{Theorem} \label{HPRY-thm-strip-num-unchanged-S}
If $G$ is a graph with leaf $u \in V(G)$ and $v \in V(G)$ is the neighbor of $u$, then $Z^-(G - \{u,v\}) = Z^-(G)$.
\end{Theorem}
\begin{proof}
Suppose that $\widetilde{\mathcal{B}} = (\widetilde{\mathcal{D}}, \widetilde{\mathcal{L}})$ is an optimal skew zero forcing set for $\widetilde{G} = G - \{u,v\}$ and let $\mathcal{D} = \widetilde{\mathcal{D}}$, $\mathcal{L} = \widetilde{\mathcal{L}} \cup \{ u\}$, and $\mathcal{B} = (\mathcal{D}, \mathcal{L})$. Carry out the forcing process on $G$ using $\mathcal{B}$ for the initial coloring, starting with $u \force v$. Since $v$ is then dark blue, it does not affect the ability of its neighbors to force. Thus the forcing process on $G$ can be continued until $\widetilde{G}$ is forced, since $\widetilde{\mathcal{B}} = \mathcal{B} \setminus \{u\}$ is a skew zero forcing set for $\widetilde{G}$. The final force can then be $v \force u$, which forces $G$, so $\mathcal{B}$ is a skew zero forcing set for $G$ with $Z^-(\widetilde{G})$ dark blue vertices. Thus $Z^-(G) \leq Z^-(\widetilde{G})$.

Now suppose that $\mathcal{B} = (\mathcal{D}, \mathcal{L})$ is an optimal skew zero forcing set for $G$; we consider three cases. As before, $\widetilde{G}$ will denote $G - \{u,v\}$.

First, if $u \in \mathcal{L}$, then $v$ is white by Lemma \ref{HPRY-lem-leaf-properties-S} and $u \force v$ can be taken as the first step of the forcing process. Without loss of generality, we can assume that $v \force u$ is the last step of the forcing process. By continuing the forcing process, we will color $\widetilde{G}$ completely dark blue, since $\mathcal{B}$ is a skew zero forcing set for $G$ and $v$ cannot force any vertex in $\widetilde{G}$; thus $\mathcal{B} \setminus \{u\}$ is a skew zero forcing set for $\widetilde{G}$ with $Z^-(G)$ dark blue vertices, so $Z^-(\widetilde{G}) \leq Z^-(G)$.

Next, suppose that $u \in \mathcal{D}$; again, by Lemma \ref{HPRY-lem-leaf-properties-S}, $v$ is white and $u \force v$ can be chosen as the first step of the forcing process. If $v$ never subsequently forces any of its other neighbors, then $\mathcal{B}$ is not optimal, since $u$ could have been chosen as a light blue vertex instead of a dark blue vertex (and then $v \force u$ could be the final step in the new forcing process). Thus $v$ must eventually force one of its neighbors, say $w$. It must be the case that at that stage all neighbors of $v$ (except $w$) are colored dark blue, and since $v$ is itself dark blue it did not affect any of the forces that led to this state. Therefore, if we let $\widetilde{\mathcal{D}} = (\mathcal{D} \setminus \{u\}) \cup \{w\}$ and $\widetilde{\mathcal{B}} = (\widetilde{\mathcal{D}}, \mathcal{L})$, we will have a set containing $Z^-(G)$ dark blue vertices that can color all of $\widetilde{G}$ dark blue. We see that $Z^-(\widetilde{G}) \leq Z^-(G)$.

Lastly, suppose that $u$ is white, so $v$ is not dark blue by Lemma \ref{HPRY-lem-leaf-properties-S}. There is a point in time after which $v$ will be dark blue; all forces prior to this time (except possibly $v \force u$ in the case where $v$ is light blue) do not involve $v$ in any way, and all forces after this time (except possibly $v \force u$) can be performed regardless of the presence of $v$, as it is dark blue. Let $\widetilde{\mathcal{B}} = (\mathcal{D}, \widetilde{\mathcal{L}})$, where $\widetilde{\mathcal{L}} = \mathcal{L} \setminus \{v\}$ if $v \in \mathcal{L}$ and $\widetilde{\mathcal{L}} = \mathcal{L}$ otherwise.  Then $\widetilde{\mathcal{B}}$ can completely force $\widetilde{G}$, so $Z^-(\widetilde{G}) \leq Z^-(G)$.
\end{proof}
%%%%%%%%%%%

Motivated by this result, we present a \defbf{leaf-stripping algorithm} that can be used to reduce a graph $G$ to a smaller graph with the same skew zero forcing number. This algorithm is a modification of Algorithm 3.16 in \cite{HPRY-GHHHJKM-mr0}.

% Hack to set the algorithm counter to the current Theorem counter value
% so that the algorithm is numbered correctly
\setcounter{algocf}{\value{Theorem}}

%%%%%%%%%
\vspace{6pt}
\begin{algorithm2e}[H]
\DontPrintSemicolon

\KwIn{Graph $G$}
\KwOut{Graph $\hat G$ with $\delta(\hat G) \neq 1$, or $\hat{G} = \emptyset$}

\BlankLine

$\hat G := G$\;

\While{{\rm $\hat G$ has a leaf $u$ with neighbor} $v$}{
	
$\hat G := \hat G - \{u,v\}$\;
}
\Return{$\hat G$}\;
\caption{Leaf-stripping algorithm}
\label{HPRY-alg-leaf-stripping-S}
\end{algorithm2e}
%%%%%%%%%%

% Increment the Theorem counter to compensate for the algorithm's number
\stepcounter{Theorem}

\vspace{6pt}

%%%%%%%%%
\begin{Theorem} \label{HPRY-zm-0-characterization-S}
Let $G$ be a graph and let $\hat{G}$ be the graph returned by Algorithm \ref{HPRY-alg-leaf-stripping-S}. Then
\begin{enumerate}[i.]
\item $Z^-(G) = Z^-(\hat{G})$; and
\item $Z^-(G) = 0$ if and only if $\hat{G} = \emptyset$.
\end{enumerate}
\end{Theorem}
\begin{proof}
The first claim follows by repeated application of Theorem \ref{HPRY-thm-strip-num-unchanged-S}. As a result, if $\hat{G} = \emptyset$, then $Z^-(G) = 0$, which proves one direction of the second claim. For the other direction, suppose that Algorithm \ref{HPRY-alg-leaf-stripping-S} does not return the empty set. If $\delta(\hat{G}) = 0$, then $1 \leq Z^-(\hat{G})$. If $\delta(\hat{G}) \geq 2$, then $1 \leq \delta(\hat{G}) - 1 \leq Z^-(\hat{G})$. In either case, $1 \leq Z^-(\hat{G}) = Z^-(G)$, which completes the proof.
\end{proof}

We immediately see that if $G$ is a graph on an odd number of vertices, then $Z^-(G) > 0$. Additionally, if $G$ is a graph with $Z^-(G) = 0$, then $G$ has a unique perfect matching; if the leaf-stripping algorithm is applied to $G$, then each removed leaf and its neighbor contribute an edge to this perfect matching. Note that having a unique perfect matching is not sufficient to guarantee that $Z^-(G) = 0$, as the next example shows.

%%%%%%%%%%%%%%%
\begin{Example}
Consider the graph $G$ shown in Figure \ref{HPRY-fig-perfect-match-ctrex}. The thick edges in the figure show the unique perfect matching for $G$, but since $\delta(G) = 2$, we have $Z^-(G) \geq 2-1 = 1$. In fact, $Z^-(G) = 1$, and the forcing set $\mathcal{B}$ shown in Figure \ref{HPRY-fig-perfect-match-ctrex} is optimal.
\end{Example}
%%%%%%%%%%%%%%%

%%%%%%%%%%%%%%%
\begin{figure}[h]
\begin{center}
\includegraphics[scale=0.8]{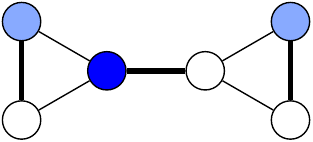}
\caption{Graph $G$ with unique perfect matching and $Z^-(G) > 0$}
\label{HPRY-fig-perfect-match-ctrex}
\end{center}
\end{figure}
%%%%%%%%%%%%%%%

%%%%%%%%%%
\begin{Remark}
If $G$ is a graph on $n$ vertices, then Algorithm \ref{HPRY-alg-leaf-stripping-S} returns the graph $\hat{G}$ in at most $\left\lfloor \frac{n}{2} \right\rfloor$ leaf-stripping steps. Theorem \ref{HPRY-zm-0-characterization-S} asserts that if $Z^-(\hat{G})$ is known, then the algorithm has computed $Z^-(G) = Z^-(\hat{G})$. In particular, if $G = T$ is a tree, then necessarily $\hat{G} = p K_1$ for some $p \geq 0$ and $Z^-(T) = p$.
\end{Remark}
%%%%%%%%%%

A natural question is whether we can prove a version of Theorem \ref{HPRY-zm-0-characterization-S} that applies to the fractional positive semidefinite forcing game. If Algorithm \ref{HPRY-alg-leaf-stripping-S} returns the empty set when applied to a graph $G$, then by Remark \ref{HPRY-rmk-zfp-leq-zm} and Theorem \ref{HPRY-zm-0-characterization-S} we have $0 \leq Z_f^+(G) \leq Z^-(G) = 0$, and so equality holds for one direction. The converse may fail, however: the graph $G$ in Examples \ref{HPRY-ex-not-AON-after-replication} and \ref{HPRY-ex-Kmn-ctrex} satisfies $Z_f^+(G) = 0$, but applying the algorithm to $G$ would return the nonempty partite set on 2 vertices. While we cannot generate a positive semidefinite analogue of Theorem \ref{HPRY-zm-0-characterization-S}, the result can still be a useful tool when Algorithm \ref{HPRY-alg-leaf-stripping-S} returns the empty set.

%%%%%%%%%%%%%%%%%%%%%%%%%%%%%
\section*{Acknowledgements}

This research has been supported in part by Iowa State University Holl Chair funds. Kevin F. Palmowski is supported by an Iowa State University Department of Mathematics Lambert Graduate Research Assistantship. David E. Roberson is supported in part by the Singapore National Research Foundation under NRF RF Award No. NRF-NRFF2013-13. Some of this work was done while Leslie Hogben and Kevin F. Palmowski were visiting the Institute for Mathematics and its Applications (IMA); they thank IMA both for financial support (from NSF funds) and for providing a wonderful collaborative research environment. The authors also extend many thanks to Steve Butler for help with computations.

%%%%%%%%%%%%%%%%%%%%%%%%%%%%%

\end{document}